\documentclass{article}

\usepackage{amsmath,amsthm,amssymb,graphicx}
\usepackage{algorithmic}
\usepackage{enumitem}
\usepackage{url}
\usepackage{hyperref}
\DeclareMathOperator{\lm}{lm}
\DeclareMathOperator{\lt}{lt}
\DeclareMathOperator{\lc}{lc}
\DeclareMathOperator{\lpos}{lpos}
\DeclareMathOperator{\POT}{POT}
\DeclareMathOperator{\Mon}{Mon}
\DeclareMathOperator{\Deg}{deg}
\DeclareMathOperator{\Diag}{Diag}

\DeclareMathOperator{\id}{id}

\DeclareMathOperator{\Ann}{Ann}

\DeclareMathOperator{\Quot}{Quot}

\newcommand{\N}{{\mathbb N}}

\newcommand{\K}{K}
\newcommand{\Z}{{\mathbb Z}}

\newcommand{\Q}{{\mathbb Q}}

\renewcommand{\d}{\partial }

\newcommand{\citep}{\cite}
\newcommand{\comment}[1]{}

\theoremstyle{definition}

\newtheorem{lemma}{Lemma}[section]
\newtheorem{proposition}[lemma]{Proposition}
\newtheorem{theorem}[lemma]{Theorem}
\newtheorem{corollary}[lemma]{Corollary}
\newtheorem{definition}[lemma]{Definition}
\newtheorem{example}[lemma]{Example}
\newtheorem{remark}[lemma]{Remark}

\newtheorem{algorithm}[lemma]{Algorithm}
\newtheorem{conjecture}[lemma]{Conjecture}

\begin{document}
 \setcounter{section}{0}


\title{Computing diagonal form and Jacobson normal form of a matrix using Gr{\"o}bner bases}

\author{Viktor Levandovskyy, Kristina Schindelar \\
Lehrstuhl D f\"ur Mathematik, RWTH Aachen, \\
Templergraben 64, 52062 Aachen, Germany \\
\texttt{[Viktor.Levandovskyy,Kristina.Schindelar]@math.rwth-aachen.de} }
\maketitle

\begin{abstract}
In this paper we present two algorithms
for the computation of a diagonal form of a matrix over non-commutative Euclidean domain
over a field with the help of Gr\"obner bases. This can be viewed as the pre-processing for the
computation of Jacobson normal form and also used for the computation of
Smith normal form in the commutative case.
We propose a general framework
for handling, among other, operator algebras with rational coefficients. 
We employ special "polynomial" strategy in Ore localizations of
non-commutative $G$-algebras and show its merits. In particular, for a
given matrix $M$ we provide an algorithm to compute $U,V$ and $D$
with fraction-free entries such that $UMV=D$ holds. The polynomial
approach allows one to obtain more precise information, than the
rational one e.~g. about singularities of the system. 

Our implementation of polynomial strategy shows very impressive
performance, compared with methods, which directly use fractions.
In particular, we experience quite moderate swell of coefficients and
obtain uncomplicated transformation matrices. This shows that
this method is well suitable for solving nontrivial practical
problems. We present an implementation of algorithms in \textsc{Singular:Plural}
and compare it with other available systems.
We leave questions on the algorithmic complexity of this algorithm open, but
we stress the practical applicability of the proposed method to a bigger class
of non-commutative algebras.  
\end{abstract}

\tableofcontents

\section{Introduction}

The existence and computation of normal forms of matrices 
over a ring is a fundamental mathematical question. 
The proof for the existence of a normal form is mainly constructive
and can be turned into an algorithm.
However, such a direct algorithm is not very efficient in general.
Computer algebra focuses its attention on this kind of problems, 
since they are of elementary interest but of high complexity. 

In that sense nearly any computer algebra system is able to compute
the 
Smith normal form for a matrix over a commutative
principal ideal domain ($\Z$ or $\K[x]$ for a field $\K$). 
There are many textbooks giving a theoretical background, like for instance \citep{Cohn, Newman}.

We present a method, which is based on Gr{\"o}bner bases. 
In \citep{Spanier}, there is a Gr{\"o}bner basis-based algorithm for the computation of Smith normal form of a matrix with entries in $K[x]$. Despite the fact that this approach seems to
be folklore, we were not able to find other references.

In this paper we consider non-commutative skew polynomial rings.
Such rings, among other, offer the possibility to describe time varying
systems in Systems and Control theory \citep{Eva}, \citep{IlchmannMehrmann}, \citep{Ilchmann}. 
Many known operator algebras can be realized as skew polynomial
rings or solvable polynomial rings \citep{Kredel}, some of them can be realized even as much 
easier Ore algebras \citep{ChyzakSalvy, CQR}. However, general solvable polynomial rings are hard to
tackle constructively (say, in a computer algebra system), while the class of Ore algebras
of \citep{ChyzakSalvy, CQR} is indeed restrictive. Based on the PBW algebras \citep{BGV} also known as $G$-algebras \citep{Viktor, Plural}, in Section \ref{Algebras} we propose a new class of univariate skew polynomial rings, which are obtained as Ore localizations of $G$-algebras. This framework is powerful and convenient at the same time. Moreover, it is more general than the class of Ore algebras (with defining endomorphism $\sigma$ being an automorphism) and allows algorithmic treatment of modules. 
In Proposition \ref{R_HIB} and Theorem \ref{BOre} several nice properties of such algebras (among other, these algebras are Noetherian domains with PBW basis) are established. We stress, that the computations in these algebras, especially Gr\"obner bases for modules, are algorithmic and, moreover, they can be done without using explicit fractions. 
It is important, that such algebras and computations in them can be realized in any computer algebra system, which can handle $G$-algebras or polynomial Ore algebras.

In \citep{CulianezQuadrat}, applications to systems of partial differential equations are shown and several concrete examples are introduced.
We generalize the idea, originating from \citep{Spanier}, to use Gr\"obner bases in computation
of normal forms for matrices. 
The crucial improvement is introduced in Section \ref{SecPoly}, where we show how to handle the problem in a completely fraction-free polynomial framework. 

We point out advantages of the polynomial strategy and illustrate
some of them with interesting examples in the Section \ref{Impl&Ex}, 
where we compare our implementation with other available packages.
In particular, we do comparisons with the implementation 
of algorithms, which use fractions directly.
Notably, in many examples our approach
delivers much more compact results with small coefficients.

The non-commutative analogue to the Smith form over a principal
ideal domain is the Jacobson form \citep{Jacobson} ,\citep{Cohn}. However, since the normal form problem is hard in general,
we propose the notion of a weak Jacobson form, that is a diagonal matrix, where 
the units on the diagonal will not be necessarily generated. 
Otherwise the advantage of the polynomial strategy is disturbed.
Instead, we propose the splitting of the whole process of obtaining a (strong) normal form
into the computation of a diagonal form and the computation of stronger diagonal form
from a given diagonal one. The latter, as we show in \ref{remJF}, \ref{cexShift} and \ref{qWeyl}
depends heavily on the domain one computes in, while the first algorithm is very general.

Our implementation (of weak Jacobson and Smith forms) is realized as the library \texttt{jacobson.lib} \citep{Jacobsonlib} for the computer algebra system {\sc Singular}:{\sc Plural} \citep{Singular, Plural}, which is freely available. The library has been already incorporated into the official distribution of \textsc{Singular} version 3-1-0. 

\section{Algebras, Localizations and their Properties}
\label{Algebras}

The framework of this paper is based on skew polynomial rings that are principal ideal domains. 
An important subclass of skew polynomial rings constitute
so-called polynomial Ore rings. They are non-commutative rings 
possessing an endomorphism $\sigma$ and a $\sigma$-derivation to define the commutation rule of two elements, that is giving the extension from commutative polynomial ring to non-commutative.
This kind of rings is used in analyzing the structure of analytic equations, 
like linear ordinary or partial differential equations or partial shift or difference equations 
with rational or polynomial coefficients, see Example \ref{interestingOreAlg}.
The name is inspired by {\O}ystein Ore, who introduced and studied this kind of rings. 
These rings were also studied, for instance in \citep{ChyzakSalvy} and \citep{McConnellRobson}. \\
Let $K$ be a field and $A$ be a $K$-algebra.
Further let $\sigma: A \rightarrow A $ be a ring endomorphism. Then 
the map $\delta : A \rightarrow A$ is called 
$\sigma$-\textbf{derivation}, if $\delta$ is $K$-linear and satisfies the skew Leibniz rule 
\begin{displaymath}
 \delta(a b) = \sigma(a) \delta(b) + \delta(a) b \; \mbox{ for all } \;a, b \in A.
\end{displaymath}
For a $\sigma$-derivation $\delta$ the ring 
$A[\partial;\sigma, \delta] $ consisting of all polynomials 
in $\partial$ with coefficients in $A$ with the usual addition 
and a product defined by the commutation rule
\begin{displaymath}
\partial a = \sigma(a) \partial + \delta(a) \; \mbox{ for all } \; a \in A
\end{displaymath}
is called \textbf{skew polynomial ring} or an \textbf{Ore extension} of $A$ with $\d$ subject
to $\sigma, \delta$. \\

It is easy to see, that any non-zero element $a\in A[\partial;\sigma, \delta]$ can be written as 
$a=a_n\partial^n + \dots + a_1\partial + a_0$, where $n \in N_0$ 
and $a_i \in A$. We call $n$ the \textbf{degree} of $a$, sometimes it
is also called the \textbf{order} of $a$.

In describing $K$-algebras via finite sets of generators $G$ and relations $R$, we write
$A = K\langle G \mid R \rangle$. It means that $A$ is a factor algebra of the free
associative algebra, generated by $G$ modulo the two-sided ideal, generated by $R$.
Hence yet another notation is $A = K\langle G \rangle / \langle R \rangle$.

\begin{example}\
\begin{itemize}
\item Defining $\sigma:=\id_A$ and $\delta:=0$ we see, 
that $(K[x_1, \dots, x_n]) [\partial;\sigma, \delta] = K[x_1, \dots, x_n, \partial]$ 
and $K(x_1, \dots, x_n) [\partial;\sigma, \delta] = K(x_1, \dots, x_n)[\partial]$.
\item Let $A=K[x] $ for a field $K$ of characteristic 0, 
$\sigma:=\id_{K[x]}$ and $\delta :=\frac{\partial}{\partial x}$. 
\begin{displaymath}
 W_1(K):=K[x][ \partial; \id_{K[x]},  \frac{\partial}{\partial x}]
= K\langle x, \d \mid \partial x = x \partial +1 \rangle
\end{displaymath}
is called the first \textbf{polynomial Weyl algebra}.
\end{itemize}
\end{example}

\begin{proposition}\citep{BGV}\label{R_HIB}
\ Let $A$ be a division ring, $\sigma: A\to A$ be an endomorphism and $R=A[\partial;\sigma, \delta]$ be an Ore extension with a $\sigma$-derivation $\delta$.\\
If $\sigma$ is injective (respectively bijective), then 
\begin{itemize}
\item (PID) $R$ is a left (resp. right) principal ideal domain.
\item (Bezout's Theorem) for any non-zero $a,b\in R$ there exists the right (resp. left) greatest common divisor $g_r$ (resp. $g_{\ell}$) of $a,b$ and there exist $s,t \in R$, such that $g_r = sa+ tb$ (resp. $s',t'$, such that $g_{\ell} = as' +bt'$).
\item (ED) $R$ is a left (resp. right) Euclidean domain.
\end{itemize}
\end{proposition}

Hence, when $\sigma$ is bijective, there are left and right Euclidean division algorithms.

In the next example we enlist some interesting skew polynomial rings 
(which are Ore algebras indeed, see \citep{ChyzakSalvy}). These rings are of great interest in
applications, all of them can be addresses with
our implementation, see Section
\ref{Impl&Ex}.

\begin{example}\label{interestingOreAlg}
Let $A=K(x)$, where $K$ is a field of characteristic $0$.
\begin{itemize}
\item Let $\sigma:=\id_{K (x)}$ and  $\delta :=\frac{\partial}{\partial x}$. Then
\begin{displaymath}
 B_1(K):=A[ \partial; \id_{K(x)},  \tfrac{\partial}{\partial x}]
= K(x)\langle \d \mid \partial x = x \partial +1 \rangle
\end{displaymath}
is called the first \textbf{rational Weyl algebra}.  \\

\item The first \textbf{rational difference algebra} is defined by
$$ \mathcal{S}_1 :=A \left[ \Delta; \sigma, \delta \right]
= K(x)\langle \Delta \mid \Delta x = x \Delta + \Delta + 1 \rangle,
$$
where $\sigma (p(x)) =  p(x + 1)$ and $\delta(p)=\sigma(p)-p$ for all 
$p \in K(x)$.
\end{itemize}

Let $q\not= 0$ be a unit (a parameter) in the ground field.

\begin{itemize}
\item Let $\sigma(p(x))=p(qx)$ and
$\delta :=(\frac{\partial}{\partial x})_q$, $\delta(f(x)) = \frac{f(qx) - f(x)}{(q-1)x}$.
 Then
\begin{displaymath}
 W^q_1(K):=A[ \partial; \sigma,  (\tfrac{\partial}{\partial x})_q]
= K(x)\langle \d \mid \partial x = q \cdot x  \partial +1 \rangle
\end{displaymath}
is called the first \textbf{rational $q$-Weyl algebra}.

\item The first \textbf{rational $q$-difference algebra} is defined by
$$ \mathcal{Q}:= A [\partial; \sigma, \delta]
= K(x)\langle \d \mid q \cdot x \partial + (q-1)x \rangle, $$
where $\sigma( p )=p(q x)$ and $\delta(p)= p(qx)-p(x)$. 
\end{itemize}
\end{example}


Indeed, we can work within the more general algebraic framework as follows. \\

Let $S$ be a multiplicatively closed set (see \citep{McConnellRobson}) in a Noetherian integral domain $A$,
such that $0 \not\in S$. $S$ is called an \textbf{Ore set} in $A$, if for all $s_1 \in S, a_1 \in A$
there exist $s_2\in S, a_2 \in A$, such that $a_1 s_2 = s_1 a_2$. Then one can see, that formally (that is, allowing fractional expressions) $s_1^{-1} a_1 = a_2 s_2^{-1}$ holds. 

Then one defines a \textbf{ring of fractions} or an \textbf{Ore localization} of $A$ with respect to $S$ to be a ring $A_S$ (often denoted as $S^{-1} A$) together
with an injective homomorphism $\phi: A \to A_S$, such that 
\begin{itemize}
\item[(i)] for all $s\in S$, $\phi(s)$ is a unit in $A_S$, 
\item[(ii)] for all $f \in A_S$, $f = \phi(s)^{-1} \phi(a)$ for some $a\in A, s\in S$.
\end{itemize}

The Ore property of $S$ in $A$ guarantees, that any left-sided fraction
can be written (non-uniquely!) as a right-sided fraction. Moreover, given 
$a_1,\ldots,a_m \in A$ and $s_1,\ldots,s_m \in S$, there exist
$a'_1,\ldots,a'_m \in A$ and $s' \in S$, such that
$a_i s' = s_i a'_i$ holds for each $i$. Thus there exist common right and
common left multiples.

\begin{remark}
Why such localizations are important? Among many motivating connections let us
state the following. Given an $A$-module homomorphism $\varphi: M \to N$, where
$M,N$ are finitely generated. Then, if $S^{-1} A$ exists, one has an induced
homomorphism of $S^{-1} A$-modules $S^{-1} \varphi: S^{-1} M \to S^{-1} N$. 
However, if one finds an appropriate multiplicatively closed Ore set $\tilde{S}$ in $A$ and proves that $\tilde{S}^{-1} \varphi: \tilde{S}^{-1} M \to \tilde{S}^{-1} N$ is not
an isomorphism, it implies that $M \not\cong N$ as $A$-modules. This gives 
an important tool to check the isomorphy of modules. In contrast with common
localizations of commutative ring at complements of prime ideals, we do not
know a priori for which $S$ we are looking for and how many different $S$
should we examine.

Note, that the question,
whether two modules are isomorphic, is one of the fundamental questions
in algebra. It is known to be not algorithmic in general, hence any partial
algorithmic answer to this question is of big importance.
\end{remark}

\begin{definition}
\label{GalgLie}
Let 
 $A$ be a quotient of the free associative algebra 
$\K\langle x_1,\ldots,x_n\rangle$ by the two-sided ideal $I$, generated by
the finite set $\{x_jx_i - c_{ij} x_ix_j-d_{ij} \} $ for all $ 1\leq i<j \leq n$, where
$c_{ij}\in \K^*$ and $d_{ij}$ are polynomials in $x_1,\ldots,x_n$. Without loss of generality \citep{Viktor} we can assume that $d_{ij}$ are given in terms of standard monomials $x_1^{a_1} \ldots x_n^{a_n}$.
$A$ is called a {\em $G$--algebra} \citep{LS03,Viktor}, if \\
$\bullet$ for all $ \; 1\leq i < j < k \leq n$ the expression
$c_{ik}c_{jk} \cdot d_{ij}x_k - x_k d_{ij} + c_{jk} \cdot x_j d_{ik} - c_{ij} \cdot d_{ik} x_j + d_{jk}x_i 
- c_{ij}c_{ik} \cdot x_i d_{jk}$
reduces to zero modulo $I$ and \\
$\bullet$ there exists a monomial ordering $\prec$ on $\K[ x_1,\ldots,x_n]$,
such that for each $i < j$, such that $d_{ij}\not=0$, $\lm(d_{ij}) \prec x_i x_j$ . Here, $\lm$ stands for the classical notion of leading monomial of a polynomial from $\K[ x_1,\ldots,x_n]$.
\end{definition}

We call an ordering on a $G$-algebra \textbf{admissible}, if it satisfies second
condition of the definition.
A $G$-algebra $A$ is Noetherian integral domain \citep{LS03}, hence there exists
its total two-sided ring of fractions $\Quot(A) = A_{A\setminus\{0\}}$, which is a division ring (skew field). Assume that $A$ is generated by
$x_1,\ldots,x_{n+1}$ and suppose that the set 
$\Lambda_n(A) = \{ \lambda = \{ i_1,\ldots,i_n\} \mid i_1 < \ldots < i_n, $ 
$\K\langle x_{i_1},\ldots,x_{i_n} \mid I_{\lambda} \rangle$ is a $G$-algebra$\}$ is not empty, 
where $I_{\lambda} = \{x_j x_i - c_{ij} x_ix_j-d_{ij} \mid \ i,j \in \lambda, i < j \}$.

For any $\lambda=\{i_1,\ldots,i_n\}\in\Lambda_n$, let us define
$B_{\lambda}$ to be a $G$-algebra, generated by $\{ x_{i_1},\ldots,x_{i_n} \}$. 

\begin{theorem}\label{BOre}
Let $A$ be a $G$-algebra in variables $x_1,\ldots,x_n, \d$ and assume
that $\lambda = \{x_1,\ldots,x_n\} \in \Lambda_n$. 
Moreover, let $B := B_{\lambda}$ and $B^* = B\setminus\{0\}$. 
Suppose, there exists an admissible monomial ordering $\prec$ on $A$, satisfying
$x_k \prec \d$ for all $1\leq k \leq n$. Then the following holds
\begin{itemize}
\item $B^*$ is multiplicatively closed Ore set in $A$. 
\item $(B^{*})^{-1} A$ (Ore localization of $A$ with respect to $B^*$) can be presented as an Ore extension of $\Quot(B)$ by the variable $\d$.
\end{itemize}
\end{theorem}

\begin{proof}
Since $B$ is a $G$-algebra itself, it is an integral domain, hence $B^*$ is multiplicatively closed
and does not contain zero.
Since $A$ and $B$ are $G$-algebras and $\prec$ is an admissible ordering,
for a relation $\d x_j = c_{j} x_j \d + d_{j}$ with $c_{j} \in K^*$
and a polynomial $d_{j}\in A$ holds $d_j =0$ or $\lm(d_{j}) \prec x_j \d$. Since $x_j \prec \d$,
then $x_j \d \prec \d^2$, hence 
$d_{j}$ is at most linear in $\d$.
Writing $d_{j} = a_{j} \cdot \d + b_{j}$ for $a_{j},b_{j} \in B$,
we define $c'_{j} = c_{j} x_j + a_{j}$ and thus we obtain a
relation $\d x_j = c'_{j} \d + b_{j}$, where $x_j, c'_{j}, b_{j} \in B$.

Then, by defining $\sigma(x_j) = c_{j}x_j + a_j$ and $\delta(x_j) = b_j$
for all $1\leq j \leq n$, we see, that $\sigma$ is an automorphism of $\Quot(B)$.
Thus an Ore extension $\Quot(B)[\d; \sigma, \delta]$ is indeed another 
presentation of $(B^{*})^{-1} A$ 
as soon as $B^*$ is an Ore set in $A$.

Since $\lm(d_j) = \lm(a_j \d + b_j) \prec x_j \d$, 
both $\lm(a_j) \prec x_j$
and 
$\lm(b_j) \prec x_j \d$ hold. The latter implies, that there
exist positive weights $\omega$ and $w_1,\ldots,w_n$ for variables $\{\d, x_1,\ldots,x_n\}$,
such that for $\lm(a_j) = x^{\alpha}$ and  $\lm(b_j)=x^{\beta}$ one has
$\sum_i w_i \alpha_i \leq w_j$ and $\sum_i w_i \beta_i \leq w_j + \omega$.
In particular, this can be achieved by setting $\omega$ big enough. Then we follow
the recipe from \citep{BGV} and construct a block ordering from this setting.
Consider an ordering
$\prec_{\d}$ on $A$, which is a block ordering for blocks of variables 
$\{\d\}, \{x_1,\ldots,x_n\}$. It means that $\d \gg x_j$ for all $j$, that is
the variable $\d$ is bigger than any power of $x_j$. The second block
is an ordering $\prec_B$ on $B$, for which $\lm(a_j) \prec_B x_j$ holds. For instance,
one can take $\prec_B$ to be the restriction of $\prec$ to $B$.
Then $\lm(d_j) = \max_{\prec_{\d}} (a_j \d, b_j) \prec_{\d} x_j \d$ holds,
hence $\prec_{\d}$ is admissible ordering on $A$. 
From the Proposition 28 of \citep{GML} (which holds for much more general situation),
the existence of such a block ordering as $\prec_{\d}$ implies, that the set $B^*$ is an Ore set in $A$. 
\end{proof}

\begin{remark}
\label{AB*}
Note, that by construction $A_{B^*} := (B^{*})^{-1} A$
is a Euclidean (principal ideal) domain by the 
Proposition \ref{R_HIB}.
In particular, all but one variables are invertible (we call them also \textit{rational} variables). We call
non-invertible variables \textit{polynomial}.
In a more general setting,
we like to present localizations of the type $A_{B^*}$, where $B$ is a sub-$G$-algebra of $A$,
 as a ring of solvable type \citep{Kredel} or, equivalently, as a PBW ring \citep{BGV}. In the
case of several polynomial variables, the analogue to the 
Theorem \ref{BOre} seem to be much more involved.
\end{remark}

\begin{example}
To illustrate the Theorem \ref{BOre}, consider the difference algebra $\mathcal{S}_1 := K \langle x, \Delta \mid \Delta x = x \Delta + \Delta + 1 \rangle$. Since $\Delta \prec x \Delta$ is a consequence of $1\prec x$
(we assume we are dealing with well-orderings only), $\mathcal{S}_1$ can be localized at both $K[x]^*$ and $K[\Delta]^*$. However, the algebra, associated with the operator of partial integration
$\mathcal{I}_1 := K \langle x, I \mid I x = x I - I^2 \rangle$ can be localized only at $K[I]^*$
but not at $K[x]^*$, since $I^2 \prec x I$ is a consequence of $I \prec x$ and any ordering, satisfying $x \prec I$ is not admissible for $\mathcal{I}_1$.
\end{example}

For many problems in module theory and in applications we would like to analyze
complicated problems via localizing at big subalgebras. In the situation as above,
we obtain non-commutative Euclidean domain as the result, 
hence we are interested in computing Jacobson form in this setting. One of the complications, which arise in constructive handling of objects over such algebras, is quite hard arithmetics in the skew field. Several fundamental
questions like the transformation of a left fraction into the right one (which is possible, since the Ore condition is satisfied), simplification of a one-sided fraction etc. require quite nontrivial
and complex algorithms (like computation of syzygy modules and so on) to be used, see for instance \citep{ApelDiss}. 
Even in the commutative case the computations (even with one variable) 
over a transcendental extension by several generators are still nontrivial and resource-consuming for most computer algebra systems. Hence saying ``ring $R$ is a (non-commutative) Euclidean domain'' does not automatically mean ``computations in $R$ are easy''.


 \begin{remark}
 Let us come back to the justification of terminology. Usually, speaking on 
 ``operator algebra with polynomial coefficients'', one means that one works 
 with the set of operators $\d_1,\ldots,\d_m$ over a commutative polynomial ring, say, $K[x_1,\ldots,x_n]$.

 By saying ``operator algebra with rational coefficients'' one addresses 
 an Ore extension of $K(x_1,\ldots,x_n)$ by the operators $\d_i$.

 It is important to mention, that
 $K(x_1,\ldots,x_n)$ is a localization of $K[x_1,\ldots,x_n]$ with respect to multiplicatively closed
 set $K[x_1,\ldots,x_n]\setminus\{0\}$. Thus it is enough to define an algebra with polynomial
 coefficients and then speak on different localizations of it. Therefore the notion of Ore localization reveals the origin of various ``rational'' coefficients and allows to treat different localizations (among them e.g. passage to the torus $K[x_1^{\pm 1},\ldots,x_n^{\pm 1}]\subset K(x_1,\ldots,x_n)$) uniformly.
 \end{remark}

\section{Gr{\"o}bner Bases in the Computation of a Diagonal Form}

\subsection{Yoga with Gr\"obner Bases}

Let us
give a short introduction to non-commutative 
Gr{\"o}bner basis theory, which has been studied by e.~g. \citep{Chyzak, Kredel, Viktor}. 
Suppose, that there is a $G$-algebra $R_*$ over a field $\K$, which is generated by $x_1,\ldots,x_n,\d$,
such that $R_* = A_*[\d; \sigma, \delta]$ is an Ore extension of a $G$-algebra $A_*$, generated by $\{x_i\}$.
By using the lower index $*$, we point out that we deal with structures, objects in which always
have a polynomial presentation.
A nice property of a $G$-algebra is that as a $K$-vector space it is generated by
\textbf{monomials} of $R_*$:
$$
\Mon(R_*) = \{ x_1^{\alpha_1} \cdot \ldots \cdot x_n^{\alpha_n}  \d^k \mid \alpha \in \N^n, k\in \N\}
= \{ x^{\alpha} \d^k \mid x^{\alpha} \in \Mon(A_*), k\in \N\}.$$
Based on a module ordering we define leading coefficient ($\lc$), leading monomial ($\lm$), 
leading term ($\lt $) and leading 
position ($\lpos$) notions as usual. Let $e_i:=(0, \ldots, 1, \ldots, 0)$
be the $i$-th unit vector.

In this paper
we will compute Gr{\"o}bner basis of modules over $R_*$ with respect to an monomial module ordering $\POT$ (position-over-term), defined as follows. For $r, s \in \Mon(R_*)$,
\begin{align}\label{orderingPOT}
& r e_i < s e_j \;\; \Leftrightarrow \;\; i<j \mbox{ or if } i=j \mbox{ then }  
r < s,
\end{align}
and $r<s$ with respect to an admissible
well-ordering on $R_*$, eliminating $\partial$, that is satisfying $\partial \gg x_n > \dots > x_1 \mbox{ on } R_*$. 

In $R$, a Gr\"obner basis is computed with respect to the induced POT ordering, which takes only degree of $\partial$ into account since $\Mon(R) = \{ \partial^k \mid k\in \N\}$.

We call $a\in R_*$ a \textbf{strict left (resp. right) divisor} 
of $b\in R_*$ if and only if $\exists \ f \in R_*$ such that $af = b$ (resp. $fa=b$).
Extending this notation to $R_*^p$ requires that both elements $a,b \in R_*^p$
have the same leading position. Moreover, $a$ is said to be a \textbf{proper} strict divisor of $b$, if either $b=af$ or $b=fa$ holds, where $f$ is not an unit in $R_*$. 
For two monomials $m_1,m_2 \in R_*$ we write
$m_1 \leq m_2$ for the comparison with the fixed monomial ordering. 
We say that $m_1$ \textbf{divides} $m_2$, if each exponent of $m_1$ is not greater than the
corresponding exponent of $m_2$.

\begin{definition}
Let $M$ be a left submodule of $R_*^p$ and $<$ be a monomial module ordering on $R_*^p$.
A finite subset $G \subset M$ is called a \textbf{Gr\"obner Basis} of $M$ with respect to $<$,
if for every $f \in M\setminus\{0\}$ there exists a $g \in G$, so that $\lm(g)$ divides $\lm(f)$.
\end{definition}

A Gr{\"o}bner basis $G$ is called \textbf{reduced} if and only if 
for any pair of polynomials $h \not= f \in G$, the leading monomial $\lm(h)$ does not divide
any monomial of $f$. It can be shown, that a normalized (that is with leading coefficients 1) 
reduced Gr\"obner basis is unique for a fixed ordering.
We recall the common property of a Gr{\"o}bner basis to be, in particular, a generating set.  

\begin{remark}\label{submoduleGroebner}
Let $M \subseteq R_*^p$ with a Gr{\"o}bner basis $G$ and $f \in M$. 
Define the submodule $S$ of $M$ to be generated by all $s \in G$
such that $ \lm(s) \leq \lm(f)$. Then $f \in S$. 
\end{remark}

\subsection{Working with Left and Right Modules}

\noindent
\textbf{Opposite algebra}. 
In order to work with left and right modules over an associative
$K$-algebra $A$, one has to use both $A$ and its opposite algebra $A^{op}$ in general.
Recall, that $A^{op}$ is the same vector space as $A$, endowed with the opposite
multiplication: $\forall \ a,b \in A^{op}$, 
$a \star_{A^{op}} b = b \cdot a$.
A natural opposing map makes from a right (resp. left) $A$-module 
a left (resp. right) $A^{op}$-module. 
There is an algorithmic procedure to set up an opposite algebra to a given
$G$-algebra, see \citep{Viktor}.\\

\noindent
\textbf{Involutive anti-automorphism}. Alternatively, for ``swapping sides'' one can employ an
anti-automorphism $\theta$ of $A$ , that is 
a $K$-linear map, which obeys $\theta(a b) = \theta(b) \theta(a)$ for all $a,b \in A$,
which is involutive, that is $\theta^2 = \id_A$. Often such an anti-automorphism is called \textbf{involution}.
In classical operator algebras, particularly simple involutions are known \citep{CQR}.
Moreover, it is possible to determine linearly presented involution of a $G$-algebra via an algorithm (Levandovskyy et~al., unpublished, see \textsc{Singular} library \texttt{involut.lib} \citep{Involutlib} for an implementation). 
A constructive advantage of using involution versus using opposite algebra lies in the fact, that
one does not need to create opposite algebra and make to an object its opposite.
Instead, we apply an involution to an object and remain in the same ring. One
application of involution means that the object we deal with change its side from
left to right or vice versa. 

An involution can be defined on matrices as follows. Let $\theta:A\to A$ 
 be an involution as above.
 We define the map
 \begin{displaymath}\label{transposition}
   \widetilde{\theta}:A^{p \times q} \rightarrow A^{q \times p}, \;\;\; M 
 \mapsto ( \theta(M))^{T},
 \end{displaymath}
 where $M^T$ is the transposed matrix of $M$ and $\theta(M) =[ 
 \theta(M_{ij}) ]$ for $1\leq i \leq p$ and $1\leq j \leq q$.\\
One can easily show that $ (\theta(B \cdot C))^T = (\theta (C))^T \cdot 
(\theta(B))^T$  for $B \in A^{p \times q}$, $C \in A^{q \times k}$. Applied twice, we 
 get $B\cdot C$ back. \\

\noindent
\textbf{Diagonalization}. Let $R$ be a $\K$-algebra and a non-commutative Euclidean PID. 
Recall, that a matrix $U \in  R^{ p \times p}$ is called \textbf{unimodular} if and only 
if there exists $U^{-1} \in  R^{ p \times p}$ such that $U U^{-1}= U^{-1}U = \id_{p \times p}$.
Let $M \in R^{p \times q}$ and assume, without loss of generality, that $p>q$. 
Then one can show, that there exist unimodular matrices $U \in R^{ p \times p}$ and 
$V \in R^{q \times q}$ such that 
\begin{displaymath}
 U M V = \left[ 
\begin{array}{ccc} 
m_1 & & 0 \\  
& \ddots &  \\ 
0 & & m_q \\ 
 & 0_{p-q} & 
\end{array} \right].
\end{displaymath}

There are several ways to prove this statement, all based on the Euclidean (and thus PID) property of the underlying ring. From now on, we assume that $R$ is a localization of a $G$-algebra as in Remark \ref{AB*}.
We present algorithms to obtain diagonal form together with unimodular transformation matrices via Gr{\"o}bner bases. 
The main idea about the computation is the sequential alternation
between the computation of a reduced Gr{\"o}bner basis of the submodule, generated
by, say, the rows of a matrix and acting by the involution $\widetilde{\theta}$ on a submodule.
In the PhD thesis \citep{Spanier} this idea was applied to $K[x]$ (of course, without using 
an involution $\theta$, which is superfluous in that case) in order to compute a Smith normal form. \\

In the following, by ${}_{R}M$ we denote the left $R$-module generated by the rows of 
a matrix $M$.
Further on, by $\mathcal{G}({}_{R}M)$ we denote the reduced left Gr{\"o}bner basis of
the submodule, generated by ${}_{R}M$ with respect to the module ordering (\ref{orderingPOT}).

For the $i$-th row of a matrix $M$ we write $M_i$ and 
$M_{ij}$ stands, as usual, for the entry in the $i$-th row and $j$-th column.
With respect to the context we identify $\mathcal{G}({}_{R}M)=\{g_1, \dots, g_m\}$ 
with the matrix $[g_1^t, \dots, g_m^t ]^t$.
Define the \textbf{degree} of an element $0 \neq m\in R^{1 \times q}$ to be the degree of the 
corresponding leading monomial, that is, $\deg(m):=\deg(\lm(m))$. Since $R$ is a PID, 
this degree measures the highest exponent in the variable $\partial$. Following the
standard convention, $\deg(0) = -\infty$.
Note that the elements of $\mathcal{G}({}_{R}M)$ have pairwise distinct leading monomials,
since they form a reduced Gr\"obner basis. 
In a reduced Gr\"obner basis $\lm(\mathcal{G}({}_R M)_i)$ $\mid$
$\lm(\mathcal{G}({}_R M)_j)$ 
if and only if $\mathcal{G}({}_R M)_i = \mathcal{G}({}_R M)_j$.

\begin{lemma}\label{triangular}
Order a reduced Gr\"obner basis in such a way, that
$\lm(\mathcal{G}({}_R M)_1) < \dots < \lm(\mathcal{G}({}_R M)_m)$.
Then $$\left[ \begin{array}{c} \mathcal{G}({}_R M)_1 \\ \vdots \\ \mathcal{G}({}_R M)_m \end{array}\right]$$
is a lower triangular matrix.
\end{lemma}
\begin{proof}
Suppose the claim does not hold. Then there exists $\mathcal{G}({}_R M)_i$ and 
$\mathcal{G}({}_R M)_j$ with $\lpos(\mathcal{G}({}_R M)_i)
=\lpos(\mathcal{G}({}_R M)_j)$ for $i<j$. 
Thus $\lm(\mathcal{G}({}_R M)_i)=\partial^{\alpha_i}e_k$ and 
$\lm(\mathcal{G}({}_R M)_j)=\partial^{\alpha_j}e_k$ such that
$\alpha_i<\alpha_j$. But then evidently 
$\lm(\mathcal{G}({}_R M)_i)$ divides $\lm(\mathcal{G}({}_R M)_j)$, which is a contradiction to 
$\mathcal{G}({}_{R}M)$ being reduced.
\end{proof}

Due to the previous lemma, we may assume without loss of generality, that the matrix $\mathcal{G}({}_R M)$ is lower triangular. Since $R$ is an integral domain,
we define the rank of a matrix $M$ to be the rank of $M$ over the field of fractions of $R$.
Now, let us assume that $p=q$ and $M$ is of full rank, that is row and column ranks of $M$ are equal to $p$. 
The non-square case will be discussed in Remark \ref{nonSquare}.

\begin{lemma}\label{columnGenerator}
Let $\mathcal{I}$ denote the left ideal generated by the elements in the last column of
$\widetilde{\theta}(\mathcal{G}({}_R M))$, that is, by $
\theta(\mathcal{G}({}_R M)_{p1}), \dots, \theta(\mathcal{G}({}_R M)_{pp})$. 
Then 
$$
\mathcal{I}= {}_R \langle \,\mathcal{G}( {}_R\widetilde{\theta}( \,
\mathcal{G}( {}_R M )\, ) \, )_{pp} \, \rangle .$$
\end{lemma}
\begin{proof}
 Note, that due to Lemma \ref{triangular}
\begin{small}
\begin{displaymath}
\underbrace{\left[ \begin{array}{ccc}
* & & \\
\vdots & \ddots & \\
\mathcal{G}( {}_R M )_{p1} & \cdots &\mathcal{G}( {}_R M )_{pp}
\end{array}\right] }_{\mathcal{G}({}_R M)}
\stackrel{\widetilde{\theta}}{\rightsquigarrow} 
\left[ \begin{array}{ccc}
 & & \theta(\mathcal{G}( {}_R M )_{p1}) \\
 & \rotatebox{75}{$\ddots$} & \vdots \\
*&\cdots&\theta(\mathcal{G}( {}_R M )_{pp})
\end{array}\right] 
\stackrel{\mathcal{G}}{\rightsquigarrow} 
\left[ \begin{array}{ccc}
*& &\\
\vdots&\ddots&\\
*&\cdots&\mathcal{G}( \, {}_R\widetilde{\theta}( \,\mathcal{G}( {}_R M )\, ) \, )_{pp}
\end{array}\right]. 
\end{displaymath}
\end{small}
According to the definition of $\mathcal{G}$ the left ideal generated by
$\mathcal{G}(_R \widetilde{\theta}( \,\mathcal{G}( {}_R M )\, ) \, )_{pp}$
coincides with ${}_R\langle \theta(\mathcal{G}( {}_R M )_{p1}), \dots, 
\theta(\mathcal{G}( {}_R M )_{pp}) \rangle$. 
\end{proof}

Now we can formulate the algorithm that yields the desired diagonal form.

\begin{algorithm}[\texttt{Diagonalization with Gr\"obner Bases}]  \label{diagonal}
\begin{algorithmic}
  \STATE 
	\REQUIRE $M \in R^{g \times g}$ of full rank, $\widetilde{\theta}$ involution as above.
	\ENSURE Matrices $U, V, D \in R^{g \times g}$, such that 
\[
U,V \text{ are unimodular and } U \cdot M \cdot V = \Diag ( r_1, \ldots, r_g ) = D.
\]
\STATE  $M^{(0)} \leftarrow M$, $U \leftarrow \id_{g \times g} $, $V \leftarrow \id_{g \times g} $ 
\STATE  $i \leftarrow 0$
  \WHILE{ ($M^{(i)}$ is not a diagonal matrix { \bf or } $i\equiv_2 1$) }
               \STATE   $i \leftarrow i+1$
               \STATE   Compute $U^{(i)}$ such that $U^{(i)} \cdot M^{(i-1)} = \mathcal{G}({}_R M^{(i-1)})$
               \STATE   $M^{(i)} \leftarrow \widetilde{\theta}(\mathcal{G}({}_R M^{(i-1)}))$
               \IF{ ($i \equiv_2 0$) }
                     \STATE $ V \leftarrow V \cdot \widetilde{\theta}(U^{(i)}) $
               \ELSE  \STATE{$ U \leftarrow U^{(i)} \cdot U $}
               \ENDIF
   \ENDWHILE
\RETURN $(U,V,M^{(i)})$
\end{algorithmic}
\end{algorithm}

\begin{theorem}\label{DiagonalForm}
The Algorithm \ref{diagonal} terminates and it is correct.\\
That is, for $M\in R^{g\times g}$, let $M^{(i)}$ denote the matrix we get after the $i$-th execution of the \textbf{while} loop. Then there exists 
an element $k \in \N$ such that $M^{(k)}$ is a diagonal matrix.
If $k$ is odd, then the \textbf{while} loop is repeated just one more time
(define $l:=k + (k\mod 2)$ in this case). 
The matrices $U, V$ obtained in the last loop are unimodular and satisfy $ U M V = \Diag(m_1, \dots, m_g)$.
\end{theorem}
\begin{proof}
We prove the claim by induction on $g$, the size of the
square matrix $M$. For $g=1$ there is nothing to show. 
Using Lemma \ref{columnGenerator}, the equality
$ {}_R\langle \theta((M^{(i+1)})_{gg}) \rangle = $
${}_R\langle (M^{(i)})_{1g}, \ldots ,(M^{(i)})_{gg} \rangle $ holds.
Hence we get 
$${}_R\langle (M^{(i)})_{gg} \rangle \subseteq {}_R\langle \theta((M^{(i+1)})_{gg}) \rangle \;\;\;\;
\mbox{ for all } i.$$
Note that $\theta$ preserves the degree. Then the previous inclusion implies by degree arguments that
${}_R\langle (M^{(r)})_{gg} \rangle = {}_R\langle(M^{(r+1)})_{gg}  \rangle$ for some $r$.
Using Lemma \ref{columnGenerator} and $(M^{(r)})_{gg}\not=0$ (since $M$ is of full rank), 
we obtain that $(M^{(r)})_{gg}$ is a strict left divisor of $(M^{(r)})_{ig}$ for each $1 \leq i \leq g-1$.
Then the definition of $\mathcal{G}$ yields that
 \begin{align*}
 M^{(r+1)} = 
 \left(
 \begin{array}{cccc}
  &  & & 0 \\
  & M' &  & \vdots  \\
  &  &  & 0\\ 
 0 & \hdots &  0 &  (M^{(r+1)})_{gg}  \\ 
 \end{array}\right)\mbox{.}
 \end{align*},
or, in a different notation, $ M^{(r+1)} = M' \oplus (M^{(r+1)})_{gg}$, that is $M^{(r+1)}$ is a block matrix.

The $(g-1) \times (g-1)$ matrix $M'$ can be transformed to a diagonal matrix
via unimodular operations by induction. It remains to consider the transformation 
matrices $U$ and $V$. For each $i \in \N$, after executing the \textbf{while} loop
$i$ times, we obtain 
\begin{displaymath}
\left\{ \begin{array}{ll} M^{(i)} = U^{(i-1)} \cdot U^{(i-3)} \cdots U^{(1)} \cdot \; M \; \cdot \widetilde{\theta}(U^{(2)}) \cdot 
\widetilde{\theta}(U^{(4)}) \cdots \widetilde{\theta}(U^{(i)}), & \mbox{ if } i \mbox{ is even }
\\ 
 M^{(i)} = U^{(i-1)} \cdot U^{(i-3)} \cdots U^{(1)} \cdot \; \widetilde{\theta}(M) \; \cdot \widetilde{\theta}(U^{(2)}) 
\cdot \widetilde{\theta}(U^{(4)}) \cdots \widetilde{\theta}(U^{(i)}), &  \mbox{ if } i 
\mbox{ is odd, } \end{array} \right.
\end{displaymath}
which completes the proof.
\end{proof}

\begin{remark}\label{nonSquare}
In order to extend Theorem \ref{DiagonalForm} and Algorithm \ref{diagonal} 
to non-square and non-full rank matrices, we need to add suitable syzygies to $U$ respectively $V$ and zero rows respectively columns to the diagonal matrix, in order to maintain the initial size of $M$. 
For a computational solution it is sufficient to extend Algorithm \ref{diagonal} in the 
following way. Let $M^i \in R^{s \times t}$ where either $s=p, t=q$ or $s=q, t=p$ in the $i$-th
while loop. Instead of computing $U^i$, satisfying $U^i \cdot M^{i-1} = \mathcal{G}({}_R M^{i-1})$,
we compute $\mathcal{G}({}_R\tilde{M})$ for the extended matrix $\tilde{M}:=[ \id_{s \times s} M^{i-1}]$. 
Such $\tilde{M}$ is obviously a full row rank matrix. 
Defining $U^i:=[\mathcal{G}({}_R\tilde{M})^{T}_1, \dots, \mathcal{G}({}_R\tilde{M})^{T}_s]^T 
\mbox{ and } M^i:=[\mathcal{G}({}_R\tilde{M})^{T}_{s+1}, \dots, \mathcal{G}({}_R\tilde{M})^{T}_t]^T$, it is easy to see that $ U^i M^{i-1}=M^i$.
The matrix $M^i$ consists of the rows of $\mathcal{G}({}_R M^{i-1})$ and additional zero
rows, such that $M^i \in R^{s \times t}$.
\end{remark}

\subsection{Polynomial Strategy}\label{SecPoly}

We are given a matrix $M$ over a non-commutative Euclidean domain $R$. In this
section, we show our main approach of this chapter. We introduce a method
that allows to execute the Algorithm \ref{diagonal} in a completely 
polynomial (that is, fraction-free) framework. The idea comes from
the commutative case and was appeared e.~g. in \citep{GTZ}.

Let $A_*$ be a $G$-algebra and $A=\Quot(A_*)$. 
Moreover, let $R=A[\partial; \sigma, \delta]$, such that $R_* = A_*[\partial; \sigma, \delta]$ is a $G$-algebra. 
Evidently $R_* \subseteq R$, since $A_* \subseteq A$.
Without loss of generality, we suppose that $M$ does not contain a zero row.\\
We define the \textbf{degree} of an element in $R_*$ (or $R_*^{1 \times g}$)
to be the weighted degree function with
weight $0$ to any generator of
$A_*$ and weight $1$ to $\partial$. Thus this weighted degree of $f\in R_*$
coincides with the degree of $f$ in $R$.
Such degree is clearly invariant under the multiplication of elements in $A_*$.

\begin{lemma}\label{inMod}
Let $M \in R^{g \times q}$. 
Then there exists an algorithm to compute a $R$-unimodular matrix $T\in R_*^{g \times g }$ 
such that $TM \in {R_{*}}^{g \times q}$. 
\end{lemma}
\begin{proof}
If $M \in {R_{*}}^{g \times g }$, there is nothing to do. Suppose that $M$ contains
elements with fractions.
At first, we show how to bring two fractional elements $a^{-1}b, c^{-1}d$ for
$a,c \in A_*$, $b,d \in R_*$ to a common left denominator, cf. \citep{ApelDiss}.
For any $h_1,h_2 \in A_*$, such that $h_1 a = h_2 c$, it is easy to see that
\[
(h_1 a)^{-1} (h_1 b) = a^{-1} h_1^{-1} h_1 b = a^{-1}b \text{ and }
(h_1 a)^{-1} (h_2 d) =  (h_2 c)^{-1}(h_2 d) = c^{-1}d,
\]
hence $(h_1 a)^{-1} = a^{-1} h_1^{-1} = (h_2 c)^{-1}$ is a common left denominator.
Analogously we can compute a common left denominator for any finite
set of fractions.
Let $T_{ii}$ be a common left denominator of non-zero elements from the 
$i$-th row of $M$, then $TM$ contains no fractions. Moreover, $T$ is a diagonal matrix
with non-zero polynomial entries, so it is $R$-unimodular.
\end{proof}

\begin{remark}
Note that the computation of compatible factors $h_i$ for $a_1,a_2 \in A_*$
can be achieved by computing syzygies, since 
$\{(h_1,h_2)\in A_*^2 \mid h_1 a_1 = h_2 a_2 \}$ is exactly
the module $Syz(a_1,-a_2)\subset A_*^2$. The factors $h_i$ for more $a_i$'s can be obtained as well.
\end{remark}

Define $M_*:=T M \in R_*^{p \times q}$ using the notation of Lemma~\ref{inMod}.
Then the relations ${}_{R_*}M_* \subseteq {}_{R}M$ and 
${}_{R}M_*={}_{R}M$ hold obviously.
Thus whenever we speak about a finitely generated 
submodule ${}_{R}M \subset R^{1 \times q}$, we denote by ${}_{R}M_*$
a presentation of ${}_{R}M$ with generators contained in $R_*$.
In what follows, we will show how to find $R$-unimodular matrices $U \in R_*^{p \times p}$
and $V \in R_*^{q \times q}$ such that
\begin{displaymath}
 U (TM) V = \left[ \begin{array}{ccc} r_1 & & \\  & \ddots &  \\ 
& & r_q \\ & 0 &   \end{array} \right] \in R_*^{p \times q}.
\end{displaymath}
Since the equality $U(TM)V = (UT)MV$ holds and $UT$ is a $R$-unimodular
matrix, our initial aim follows.\\
As in the previous subsection, by $\mathcal{G}({}_{R_*}M_*)$ we denote the reduced left Gr{\"o}bner basis of the submodule ${}_{R_*}M_*$ with respect to the module ordering $<_*$
on $R_*$, which was already defined in (\ref{orderingPOT}).
Unlike the rational case, the leading monomials of elements in $R_*^{1 \times g}$
are of the form $x_1^{\alpha_1} \cdots x_n^{\alpha_n}\partial^\beta$ for 
$\alpha_i, \beta \in \N$.

\begin{remark}\label{PolyIsBetter}
Using the polynomial strategy, two improvements can be observed.
On the one hand, once we have mapped the matrix we work with from $R^{g \times q}$ to $R_*^{g \times q}$, the complicated arithmetics in the skew field of fractions is not used anymore. 
The other improvement lies in the nature of construction of normal forms
for matrices and the corresponding transformation matrices.
The naive approach would be to apply elementary operations inclusive division by invertibles
on the rows and columns, that is, operations from the left and from the right. 
Indeed, there are methods using different techniques like, for instance, $p$-adic arguments to 
calculate the invariant factors of the Smith form over integers \citep{Frank}, but
this method does not help in construction of transformation matrices.
Surely the swap from left to right has no influence in the commutative framework.
But already in the rational Weyl algebra $B_1$ (see Example \ref{interestingOreAlg}),
$\frac{1}{x}$ is an unit in $B_1$ and 
$\begin{displaystyle} \partial \tfrac{1}{x} =  \tfrac{1}{x} \partial - \tfrac{1}{x^2} \end{displaystyle}$.
Comparing the multiplication with the inverse element, that is, with $x$, we see that 
$\partial x=  x \partial + 1$ holds.
Thus a multiplication of any polynomial containing $\d$ with the element $\frac{1}{x}$ in the field of fractions  
causes an immediate coefficient
swell. Since a normal form of a matrix is given modulo unimodular operations, 
the previous example illustrates the variations of possible representations. 
In section \ref{Impl&Ex}, we present nontrivial examples. 
Especially in the Example \ref{Ex3dim}, the polynomial strategy 
dams up the coefficient increase in a very impressive way.

On the other hand, switching to the polynomial framework changes the setup.
The ring $R_*$ is not a PID anymore, what
was the essential property for the existence of a diagonal form over $R$.
In the sequel, we show how that this problem can be 
resolved by introducing a suitable sorting condition
for the chosen module ordering.
Referring to the argumentation of Remark \ref{triangular} yields 
the block-diagonal form with the 0 block above. 

\begin{footnotesize}
\begin{align}\label{BlockTriangle}
 \mathcal{G}({}_{R_*} M_*)=\left[ 
 \begin{array}{cccc}
 0 &    \dots &    \dots & 0 \\
  \fbox{*} &  &  &  \\
  \vdots &  & 0 &  \\
  * &  & \;\; &  \\
   & \fbox{*} &  \;\; &  \\ 
   & \vdots &  \;\; & 0 \;\;\;\\ 
   & * & \;\; \\
   &  & \ddots & \\
    &      &    & \fbox{*} \\
   &    *  &    & \vdots \\
   &      &    & * \\
 \end{array}
 \right].
\end{align}
\end{footnotesize}
Moreover, the rows with the boxed element have the smallest leading monomial with respect to the chosen ordering in the corresponding block. A block denotes all elements of same leading position in $\mathcal{G}({}_{R_*}M_*)$.
In Theorem \ref{polynomialBasis} we show that 
these elements indeed generate ${}_R M$, while in Lemma \ref{MinDeg} we show that these elements provide us with additional information.
However, this result requires some preparations.
\end{remark}

\begin{lemma}\label{smallestElementInPoly}
Let $P$ be $R$ or $R_*$. For $M \in P^{g \times q}$ of full rank and every $1 \leq i \leq g$, 
define $\alpha_i:=\min\{ \deg(a) \mid a \in {}_P M \setminus \{0\} \mbox{ and } \lpos(a)=i\}$. 
Then for all $1 \leq i \leq g$, there exists $h_i \in \mathcal{G}({}_P M)$ of degree $\alpha_i$ with $\lpos(h_i)=i$.
\end{lemma}
\begin{proof}
Recall that $\partial \gg x_j$ for all $j$. 
Let $f \in {}_PM$ with $\lpos(f)=i$ and $\deg(f)=\alpha_i$. 
Suppose that for all $ g \in \mathcal{G}({}_P M)$ with 
leading position $i$, $\deg(g) > \alpha_i$ holds. Since $\mathcal{G}({}_P M)$ 
is a Gr{\"o}bner basis, there exists $g \in \mathcal{G}({}_P M) $  
such that $\lm(g)$ divides $\lm(f)$. This happens if and only if $\deg(g)\leq \deg(f)$ (because
$R_*$ is a $G$-algebra and $R$ is an Ore PID), which yields a contradiction.
\end{proof}

The full rank assumption in the 
lemma guarantees the existence of $\alpha_i$ 
for each component $1 \leq i \leq g$. Note, that over $R_*$ the cardinality of
 $\{ \deg(a) \mid a \in {}_P M \backslash \{0\} \mbox{ and } \lpos(a)=i\}$ is
more than one in general, hence there might be different selection strategies.
We propose to select an element according to $\min_{<_*}$
, see Lemma \ref{MinDeg}.

\begin{corollary}\label{smallestElementInPolyMat}
Lemma \ref{smallestElementInPoly} and Lemma \ref{triangular} yield
$$ \deg(\mathcal{G}({}_{R}M)_i)= \min\{ \deg(a) \mid a \in {}_R M\setminus \{0\} \mbox{ and } \lpos(a)=i\}.$$
\end{corollary}

\begin{lemma}\label{MinDeg}
Let $\alpha_i$ be the degree of the boxed entry with leading position
in the $i$-th column, that is
$$\alpha_i:=\deg( \, \min_{<_*}\{ b \mid  b \in \mathcal{G}({}_{R_*}M_*) \mbox{ and } \lpos(b) =i \} 
\, ) .$$
Then for all $h \in {}_R M$ with $\lpos(h)=i$ we have $\deg(\lm(h)) \geq \alpha_i$.
\end{lemma}
\begin{proof}
Now suppose the claim does not hold and there is $h \in {}_R M$
with $\lpos(h)=i$ of degree smaller than $\alpha_i$. Using Lemma
\ref{inMod}, there exists $a \in A_*$ such that 
$a h \in {}_{R_*}M_*$. Then $\deg(ah) = \deg(h)$ and 
$\lpos(ah)=i$. Due to Lemma \ref{smallestElementInPoly}, $\deg(f)\geq \alpha_i$
for all $f \in {}_{R_*}M_*$ with leading position $i$, hence we obtain a contradiction.
\end{proof}

\begin{corollary}\label{DegreeInv}
Lemma \ref{MinDeg} and Corollary \ref{smallestElementInPolyMat} imply, that 
$\forall \ 1 \leq i \leq g$
$$\min\{ \deg(a) \mid a \in {}_R M\setminus \{0\} \wedge \lpos(a)=i\}=
\min\{ \deg(a) \mid a \in {}_{R_*} M_*\setminus \{0\}  \wedge \lpos(a)=i\}.$$
\end{corollary}

\begin{theorem}\label{polynomialBasis}
Let $M \in R^{g \times g}$ be of full rank.  
For each $1 \leq i \leq g$, let us define
$$b_i:= \min_{<_*}\{ b \mid b \in \mathcal{G}({}_{R_*}M_*) \mbox{ and } \lpos(b) =i \}.$$
Since $M$ is of full rank, the minimum exists for each $1 \leq i \leq g$.
Note that the set $\{b_1, \ldots, b_g\}$ corresponds to the subset of all rows 
with a boxed entry in the block triangular form \ref{BlockTriangle}.  
Moreover ${}_R\langle b_1, \dots, b_g \rangle = {}_R M$. 
\end{theorem}
\begin{proof}
Let $f \in {}_RM \backslash \{0\}$. Due to Corollary \ref{DegreeInv}, there exists $1\leq k \leq g$ such that
$\lpos(b_k)=\lpos(f)$ and $\deg(b_k)\leq\deg(f)$. Thus there exists an element
$s_k \in R$ such that $\deg( f- s_kb_k)<\deg(b_k)$. Since $f-s_kb_k \in {}_RM$,
Corollary \ref{DegreeInv} implies that we have $\lpos(f- s_kb_k)< \lpos(f)$. Iterating this
reduction leads to the remainder zero and thus $f=\sum_{i=1}^k s_ib_i$.
\end{proof}

Using the notation of the previous theorem, let
$G^*({}_R M):=\left[ b_1, \ldots, b_g \right]^T$, which is by definition a lower triangular matrix.
In the sequel, let $M\in R^{g \times g}$ be of full rank. 
Note that then obviously $G^*({}_R M)$ is a square matrix.

\begin{proposition}\label{trafoMatrix}
Suppose $M \in R^{g \times g}$ is a full rank matrix and there is $U_* \in R_*^{l \times g}$
such that $U_* M_* = \mathcal{G}({}_{R_*}M_*)$. 
Let us select the indices
\begin{align}\label{indexChoice}
\{t_1, \dots, t_g\} \subseteq \{1, \dots, l\} \mbox{ such that } 
\{(U_* M_*)_{t_1}, \dots, (U_* M_*)_{t_g}\} 
= G^*({}_R M)
\end{align}
in the notation of Theorem \ref{polynomialBasis}. Then 
$U:=[(U_*)_{t_1}, \dots, (U_*)_{t_g}]^T$ is $R$-unimodular in $R^{g \times g}$
and $U M_* = G^*({}_R M)$.
\end{proposition}
\begin{proof}
The equality $U M_* = G^*({}_R M)$ follows by the definition of $U$. Now we show that $U$ is $R$-unimodular. 
Note that ${}_R(U M_*) = {}_RG^*({}_R M) = {}_R M = {}_R M_*$
holds and $UM_* \subset R^{g \times g} \supset M_*$. Thus there exists $V \in R^{g \times g}$ such that
$M_*=V(UM_*)$. Then $VU = \id_{g \times g}$ and analogously $UV=\id_{g \times g}$ since $M$ has full row rank.
\end{proof}

\begin{lemma}\label{columnGeneratorPoly}
The equality of the following left ideals holds:
$${}_R\langle \theta(G^*( {}_{R}M)_{g1}), \dots, \theta(G^*( {}_{R}M)_{gg})\rangle=
{}_R \langle G^*(\widetilde{\theta}( G^*( {}_{R}M))_{gg} \rangle. $$
\end{lemma}
\begin{proof}
Using the argumentation given in the proof of Lemma \ref{columnGenerator}  
we obtain
$${}_R\langle \theta(G^*( {}_{R}M)_{g1}), \dots, \theta(G^*( {}_{R}M)_{gg})\rangle=
{}_R \langle \mathcal{G}(\widetilde{\theta}(G^*( {}_{R}M))_{gg} \rangle.$$
Note the module identities ${}_{R}G^*( {}_{R}M) = {}_{R}\mathcal{G}( {}_{R}M) 
\Rightarrow$
\[
\widetilde{\theta}(G^*( {}_{R}M))_{R} = 
\widetilde{\theta}(\mathcal{G}( {}_{R}M))_{R}
\Rightarrow \;
{}_{R}G^*(\widetilde{\theta}(G^*( {}_{R}M))) = 
{}_{R}\mathcal{G}(\widetilde{\theta}(\mathcal{G}( {}_{R}M))).
\]
According to the latter identity and to the fact that both
$\mathcal{G}(\widetilde{\theta}(\mathcal{G}( {}_{R}M))$ and 
$G^*(\widetilde{\theta}(G^*( {}_{R}M)))$ are lower triangular matrices, 
we obtain
${}_{R}\langle \mathcal{G}(\widetilde{\theta}(\mathcal{G}( {}_{R}M))_{gg} \rangle=
{}_{R}\langle G^*(\widetilde{\theta}(G^*( {}_{R}M))_{gg} \rangle$.
\end{proof}

Now we are ready to formulate the polynomial version of Algorithm \ref{diagonal}.


\newpage 

\begin{algorithm}[\texttt{Polynomial diagonalization with Gr\"obner Bases}]  \label{diagonalPoly}
\begin{algorithmic}
  \STATE 
	\REQUIRE $M \in R^{g \times g}$ of full rank, $\theta$ an involution on $R_*$ and $\widetilde{\theta}$ as above.
	\ENSURE $R$-unimodular matrices $U, V, D \in R_*^{g \times g}$ such that $U \cdot M \cdot V = D = \Diag ( r_1, \dots, r_g )$.
\STATE  Find $T \in R^{g \times g }$ unimodular such that $T M \in R_*^{g \times g}$
\STATE   $M^{(0)} \leftarrow T M$, \quad $U \leftarrow T $, \quad $V \leftarrow \id_{g \times g} $ 
\STATE  $i \leftarrow 0$ 
 \WHILE{$M^{(i)}$ is not a diagonal matrix { \bf or } $i\equiv_2 1$}
\STATE     $i \leftarrow i+1$
\STATE    Compute $U^{(i)}$ so that  $U^{(i)} \cdot M^{(i-1)} = \mathcal{G}({}_{R_*} M^{(i-1)}) \in R_*^{l \times g}$
\STATE     Select $\{t_1, \dots, t_g\} \subseteq \{1, \dots, l\}$ as in (\ref{indexChoice})
\STATE     $ U^{(i)} \leftarrow [(U^{(i)})_{t_1}, \dots, (U^{(i)})_{t_g} ]^T$ 
\STATE     $M^{(i)} \leftarrow \widetilde{\theta}( G^*({}_R M) )$
              \IF{$i \equiv_2 0$}
                   \STATE  $V \leftarrow V \cdot \widetilde{\theta}(U^{(i)}) $
                    \ELSE \STATE $ U \leftarrow U^{(i)} \cdot U $
              \ENDIF
   \ENDWHILE
\RETURN $(U,V,M^{(i)})$
\end{algorithmic}
\end{algorithm}

\begin{remark}
\label{unimod}
It is important to mention, that the matrices $U,V,D$ (hence the elements $r_i$ as well) have entries from $R_*$, that is polynomials. However, $U$ and $V$ are only unimodular over $R$ and, in general, they need not be unimodular over $R_*$ for obvious reasons. In each presented example we will investigate the case, whether $U$ or $V$ will be unimodular over $R_*$ as well. After all, we come up with the Conjecture \ref{conjUnimod}. 
\end{remark}

\begin{theorem}
 Algorithm \ref{diagonalPoly} terminates with the correct result.
\end{theorem}
\begin{proof}
Using Proposition \ref{trafoMatrix} and
replacing Lemma \ref{columnGenerator} by Lemma \ref{columnGeneratorPoly}
in the proof of Theorem \ref{DiagonalForm} provides the claim.
\end{proof}

\begin{example}\label{runningEx2}
Suppose $R=K(x)[\partial; \id, \frac{d}{dx}]$ and $R_*=K[x][\partial; \id, \frac{d}{dx}]$. 
Let us define an involution on $R_*$ by $\theta(\partial) = -\partial$ and $\theta(x) = x$.
Let 
\begin{displaymath}
 M=\left[ \begin{array}{cc} \partial^2-1 & \partial+1 \\ \partial^2+1 & \partial-x  \end{array}\right]
\in R^{2 \times 2}.
\end{displaymath}
Evidently $T=\id_{2 \times 2}$ and thus $M^{(0)}:=M,$ $U=V=\id_{2 \times 2}$ and $i=0$.
\begin{itemize}
 \item[1:]Since $M^{(0)}$ is not diagonal, go into the while loop
\begin{itemize}
 \item[$\bullet$] $i \leftarrow 1$. Since \begin{footnotesize}                 
$
\;\;\left[ \begin{array}{cc} -x\partial-\partial+x^2+x+1 & x\partial+\partial+x\\ 
 -\partial^2+x\partial-\partial+x+2 & \partial^2+2\partial+1 \\ \partial-x & -\partial-1 \end{array}\right]
$\end{footnotesize} $M^{(0)}=\mathcal{G}({}_{R_*}M^{(0)})$ \\

where $\mathcal{G}({}_{R_*}M^{(0)})=$\begin{footnotesize}                 
$
\;\;\left[ \begin{array}{cc} 
x^2\partial^2+2x\partial^2+\partial^2+2x\partial+2\partial-x^2-1 & 0\\x\partial^3+\partial^3+x\partial^2+5\partial^2-x\partial+3\partial-x-1 & 0\\
-x\partial^2-\partial^2-2\partial+x-1 & 1
\end{array}\right]$\end{footnotesize} and $\;\;i \equiv_2 1$\\

$M^{(1)} \leftarrow $\begin{footnotesize} $\left[ \begin{array}{cc} 
x^2\partial^2+2x\partial^2+\partial^2+2x\partial+2\partial-x^2-1 & -x\partial^2-\partial^2+x-1 \\
0 & 1 \end{array}\right]$\end{footnotesize}\\

$U \leftarrow $ \begin{footnotesize} $
\left[ \begin{array}{cc} -x\partial-\partial+x^2+x+1 & x\partial+\partial+x \\ \partial-x &  -\partial-1\end{array}\right]$                \end{footnotesize}
\end{itemize}
\item[2:] Since $M^{(1)}$ is not diagonal, go into the while loop
\begin{itemize}
\item[$\bullet$] $i \leftarrow 2$. Since
\begin{footnotesize}
$\;\;\left[\begin{array}{cc} 1& x\partial^2+\partial^2-x+1 \\ 0 & 1\end{array}\right]$\end{footnotesize}$
 M^{(1)}= \mathcal{G}({}_{R_*} M^{(1)})\;\;$ and $\;\;i \equiv_2 0$ \\

$M^{(2)} \leftarrow \begin{footnotesize} \left[ \begin{array}{cc} x^2\partial^2+2x\partial^2+\partial^2+2x\partial+2\partial-x^2-1 & 0 \\ 0 & 1 \end{array}\right]\end{footnotesize}$, \\

$V\leftarrow \begin{footnotesize} \left[\begin{array}{cc} 1 & 0\\(x+1)\partial^2+2\partial-x+1 & 1\end{array}\right]\end{footnotesize}$

\end{itemize}
\item[3:] Since $i$ is even and $M^{(2)}$ is diagonal, the algorithm returns $U$ and $V$.
\end{itemize}
And indeed, the algorithm outputs the claimed result, since
\begin{center}
$UMV=$\begin{footnotesize}$\left[\begin{array}{cc}
x^2\partial^2+2x\partial^2+\partial^2+2x\partial+2\partial-x^2-1 & 0\\0 & 1\end{array}\right]$\end{footnotesize}.
\end{center}
In view of Remark \ref{unimod}, let us analyze the transformation matrices for $R_*$-unimodularity. Indeed,
$V$ is such since it admits an inverse $V'$. On the contrary, $U$ is only unimodular over $R$ and not over $R_*$, since $U\cdot Z = W$ and $W$ is first invertible in the localization:
\[
\begin{footnotesize}
V' = \left[
\begin{array}{cc}
1 & 0 \\
-(x+1) \d^{2}+x-2\d-1 & 1
\end{array}\right], 
Z = \left[
\begin{array}{cc}
2\d+2 & (x+1)\d+x-2 \\
2(\d-x) & (x+1)\d-x^{2}-x-3
\end{array} \right],
\end{footnotesize}
\]
\[
\begin{footnotesize}
W = \left[
\begin{array}{cc}
0 & -4x^{2}-8x-4 \\
2 & 5x+5
\end{array}\right]
\end{footnotesize}.
\]
\end{example}

Algorithm \ref{diagonalPoly} can be extended to $M \in R^{g \times q}$ along the lines already 
presented in Remark \ref{nonSquare}. 

\begin{example}
By executing the algorithm in the 1st rational shift algebra $\K\langle t,S \mid St=tS+S \rangle$
on the same matrix as in the previous example, where $\partial$ is  replaced with
the forward shift operator $S$, we obtain a diagonal form \\
$\Diag \begin{footnotesize}  ( \ (t^{2}+3t+2)S^{2}+2(t+1)S-t^{2}-t+2, 1) \end{footnotesize} = $
\[
\begin{footnotesize}
\left[\begin{array}{cc}
-tS-S+t^{2}+2t & tS+S+t+2 \\
-S+t+1 & S+1
\end{array}\right]
\cdot
\left[\begin{array}{cc}
S^{2}-1 & S+1 \\
S^{2}+1 & S-t
\end{array}\right]
\cdot
\left[\begin{array}{cc}
1 & 0 \\
-tS^{2}-2S^{2}-2S+t & 1
\end{array}\right]
\end{footnotesize}.
\]
As in the previous example, it turns out that $V$ (but not $U$) is even $R_*$-unimodular.
\end{example}

\section{Implementation and Examples}
\label{Impl&Ex}

\subsection{Jacobson Form}

Let $R$ be a left and right Euclidean domain.
Inspired by the Smith form, we will focus on how to sharpen the result of the already discussed
diagonal form. 
\begin{theorem}\label{TotalDivisorForm}\citep{Cohn,Jacobson}
Every matrix $M \in R^{g \times q}$ is associated to
a certain diagonal matrix, namely $\Diag(m_1, \dots, m_\ell, 0, \dots, 0)$ such that additionally 
\begin{align}\label{RelationGeneratingNF} 
 R m_{i+1} R \subseteq m_{i} R \cap R m_{i}
\end{align}
holds for all $i=1, \dots, \min\{g,q\}-1$. 
\end{theorem}

Due to \cite[Theorem 31]{Jacobson} the elements $m_i$ are unique up to similarity.
Two elements $m_i$ and $n_i$ are called \textbf{similar} if and only if there
exist $a, b \in R$ such that 
$$ am_i = n_ib, \;\;\;\; R=aR+n_iR, \;\;\;\; R=Rb+Rm_i.$$ 
Using the notation of the previous theorem, we call 
$\Diag(m_1, \dots, m_\ell, 0, \dots, 0)$ a \textbf{Jacobson normal form} of $M$.
Note that (\ref{RelationGeneratingNF}) is hard to tackle constructively in general, 
since it requires to work with 
the intersection of a left and a right ideal. 
This difficulty disappears if $R$ has only trivial two-sided ideals, that is when $R$ is simple. Then each matrix $M$ possesses a Jacobson form $ \Diag(1, \dots, 1, m_M, 0, \dots, 0)$ with $m_M \in R$.\\

\begin{lemma}\label{OrderOfSystem}
Let $A_*$ be a $G$-algebra, $A=\Quot(A_*)$ and $R = A[\partial; \sigma, \delta]$. 
Let $U, V$ be unimodular and $a, b, c, d \in R\setminus \{0\}$ such that 
\begin{align}\label{OrderOfSystemEq} 
 U \Diag(a,b)V= \Diag(c,d).
\end{align}
Then $\deg(a)+\deg(b)=\deg(c)+\deg(d)$.
\end{lemma}
\begin{proof}
 Due to (\ref{OrderOfSystemEq}) there exists a $R$-module isomorphism 
 $$\phi: R / a R \oplus R / b R \rightarrow R / c R\oplus R / d R.$$
Since $A$ is a skew field, $\phi$ induces an $A$-vector space isomorphism.
Thus the $A$-dimensions of $R / a R \oplus R / b R$ and $R / c R \oplus R / d R$, which
are nothing else that the sums of degrees, coincide. 
\end{proof}

Of course, inductive argument implies that sums of degrees of diagonal entries of 
two diagonal presentation matrices of the same module are the same. \\

\noindent
\textbf{Jacobson normal form in the 1st Weyl algebra}. Let $R$ be the rational Weyl algebra $K(x)[\d;1,\frac{\d}{\d x}]$, which is a simple domain.

\begin{lemma}\label{TF}
Consider $a, b \in R$ with $\Deg(a) > 0$, $b \neq 0$ and $\Deg(b)\geq \Deg(a)$. Then 
there exists $i \in \left\lbrace 0, \dots, \Deg (b) - \Deg (a)+1\right\rbrace$ 
such that $a$ is not a strict right divisor of $bx^i$.
\end{lemma}
\begin{proof}
Suppose that for every $i \in \left\lbrace 0, \dots, \Deg(b) - \Deg(a)+1\right\rbrace $ there exists a $q_i \in R$ such that $b x^i=q_i a$. 
Let $b = b_{n}(x)\partial^{n} + \dots + b_1(x)\partial + b_0(x)$. Note, that for any $k\in \N$
the equality $\d^k x = x \d^k + k \partial^{k-1}$. Thus we define
$r_1 := bx-xb = \sum_{i=1}^{n} b_i(x)i\partial^{i-1}$
 with $\Deg(r_1)=n-1<\Deg(b)$ and $r_1 \not= 0$ since $\Deg(b)\geq 1$.
Since $b = q_0 a$ and $bx = q_1 a$, it follows that
$r_1 = bx- xb = (q_1 - q_0 x) a$, that is $a$ is a strict right divisor of $r_1$.
By proceeding with $bx^2$ and so on,
we obtain a sequence of non-zero polynomials $r_i$, such that
$\Deg(b) > \Deg(r_1) > \ldots $ and 
$a$ is a strict right divisor of $r_i$.
Since the degree of $r_i$ decreases exactly by 1 at each step, 
after at most $\Deg (b) - \Deg (a)+1$ iterations we obtain a polynomial of degree 
$\Deg(a) -1$, which is non-zero. Such a polynomial must contain a right factor of degree $\Deg(a)$, what is a contradiction.
\end{proof}




The Lemma (\ref{TF}) suggests an algorithm to compute the Jacobson form from a diagonal matrix over the rational Weyl algebra.
Suppose $M \in R^{g \times q}$, where $g=q=2$. The extension to $g, q \in \N$
is evident. Algorithm \ref{diagonalPoly} returns unimodular matrices $U, V$ such that $UMV=\Diag(m_1, m_2)$. 
Without loss of generality, assume that $\Deg(m_2)\leq \Deg(m_1)$. 
\begin{itemize}
 \item[1)] If $m_2$ is a unit, we get the Jacobson form just by replacing $U$ by $\Diag(1, m_2^{-1})U$. 
Otherwise, choose an exponent $i \in \N$ (it exists by the Lemma \ref{TF}) such that $m_1x^i=am_2+b$ with $\Deg(b)<\Deg(m_2)$ and $b \neq 0$. Then
$$ 
\left[ \begin{array}{cc} 1 & -a  \\ 0 & 1 \end{array}\right]\cdot
\left[ \begin{array}{cc} m_1 & 0  \\ 0 & m_2 \end{array}\right]\cdot 
\left[ \begin{array}{cc} 1 & x^i  \\ 0 & 1 \end{array}\right]=
\left[ \begin{array}{cc} m_1 & b  \\ 0 & m_2 \end{array}\right].
$$
Replace $U$ by $\left[ \begin{array}{cc} 1 & -a  \\ 0 & 1 \end{array}\right]U$ and
$V$ by $V\left[ \begin{array}{cc} 1 & x^i  \\ 0 & 1 \end{array}\right]. $
\item[2)] Now we apply Algorithm \ref{diagonalPoly} to 
the matrix $\left[ \begin{array}{cc} m_1 & r  \\ 0 & m_2 \end{array}\right]$. 
The result is then \\
$\Diag(m_1', m_2')$, where $\Deg(m_2')<\Deg(m_2)$. 
\end{itemize}
Thus, by iterating 1) and 2) we compute $U$ and $V$, such that $UMV=\Diag(1,m_M)$.

\begin{remark}
\label{remJF}
It seems to us, that the process of obtaining Jacobson normal form from an appropriate diagonal matrix can be generalized to any constructive simple Euclidean PID. Moreover, it can be applied even over non-simple domains. There, it is not guaranteed, that the result is so nice as Jacobson form, but the procedure above can simplify diagonal matrices.
\end{remark}

\begin{example}
\label{cexShift}
Over the first rational shift algebra $A = \K(t)\langle s \mid st=ts+s \rangle$ (which is a not a simple domain), we provide a counterexample for a statement, similar to \ref{TF}. Consider the $2\times 2$ diagonal matrix $D_1 = \Diag(s,s)$. Then the left module $M_1 = A^2/A^2 D_1$ (it is of dimension 2 over $\K(t)$) is annihilated by the two-sided ideal $\langle s \rangle$ and hence $D_1$ is not equivalent to a matrix of the form $D_2 = \Diag(1,p)$. If it were so, 
due to the $\K(t)$-dimension of $M_1$ and hence $M_2 = A^2/A^2 D_2$, we see that $\lm(p) = s^2$. Since $M_2 = A^2/A^2 \Diag(1,p) \cong A/Ap$, we have $\Ann_A M_2 = \langle p \rangle$.

Since it is not equal to $\Ann_A M_1 = \langle s \rangle$, $M_1\not\cong M_2$.
Hence, unlike over the Weyl algebra (or a simple domain \cite{Cohn}), there are many different types of normal forms even for diagonal matrices.
\end{example}

\begin{example}
\label{qWeyl}
Consider the rational $q$-Weyl algebra, cf. \ref{interestingOreAlg}. 
It is not simple since e.~g. the ideal $\langle (q-1) \d+ x^{-1} \rangle$ is a proper two-sided ideal.
Denote the generator by $f$, then, by the same argumentation as in the previous example we can
show, that $\Diag(f,f)$ is not equivalent to any matrix of the type $\Diag(1,g)$.
\end{example}

\noindent
Since little is known about normal forms of non-simple domains, this approach is very interesting to investigate in the future. \\



\noindent
\textbf{Cyclic vector method}.
Indeed, the existence of Jacobson form in simple Euclidean PID is very strong result. In particular, it tells us that any finitely generated module is cyclic and its presentation is a principal ideal.
The method of finding a cyclic vector in a module and obtain a left ideal, annihilating it,
is used in $D$-module theory. J.~Middeke in \citep{Middeke} did some investigations of
this question.

\begin{conjecture}
We conjecture, that the Jacobson form for, say, square matrix $M$ over a simple Euclidean domain $R$
can be computed from the given
diagonal form in the following way. 
Let $M = \Diag(m_1,\ldots,m_r)$. 
Since $\sum \deg(m_i)$ is invariant of the module $R^{r\times r}/M$, this number
can be used as a certificate for probabilistic approach. Namely,
consider polynomials $p_i$ of degree at most $\deg(m_i)-1$ with random coefficients
in $A$. Compute a generator $c\in R$ of the left annihilator ideal of a vector $[p_1,\ldots,p_r]^T$ in $R^{r\times r}/M$.
If $\deg c = \sum \deg(m_i)$,
then $\Diag(1, \dots, 1, c)$ is a Jacobson form of $M$. Otherwise one takes another set of random polynomials $p_i$ and repeats the procedure.
\end{conjecture}
\noindent
One needs the probabilistic estimations on the length of random coefficients like in \citep{Kaltofen}. 

\subsection{Examples, Applications and Comparison}

\noindent
\textbf{Implementations of Jacobson normal form.} 
To the best of our knowledge, Jacobson normal form algorithm has been
implemented in \textsc{Maple} by Culianez and Quadrat \citep{CulianezQuadrat},
by Robertz et~al. \citep{Robertz, CQR}, by Middeke \citep{Middeke}
and by Cheng et~al \citep{BCL06, CL07, DCL08}. \\

We could not locate the download version of the implementation of \citep{CulianezQuadrat}.
The packages \textsc{FFreduce} \citep{BCL06} and \textsc{Modreduce}  \citep{CL07} are
available via personal request to their authors. The implementation of J.~Middeke \citep{Middeke} was, according to its author, merely a check of ideas and was not supposed to become a freely distributed package for \textsc{Maple}. This package is able to compute in the 1st Weyl algebra with coefficients in a differential field. \\

D.~Robertz informed us, that his implementation \citep{Robertz} directly follows the classical algorithm and it has not been specially optimized. Nevertheless, in what follows, we compare our implementation with the one in the \textsc{Maple} package \textsc{Janet} \citep{Robertz} on some nontrivial examples. This package is available to general public. \\

In packages by H.~Cheng et~al. modular (\textsc{Modreduce}) and fraction-free (\textsc{FFreduce}) versions of an order basis of a polynomial matrix $M$ from an Ore algebra $A$ are implemented. In particular, such a basis is used to compute the left nullspace of  $M$, and indirectly the Popov form of $M$.  \\

\noindent
\textbf{Examples.} As we already pointed out in the introduction, behind 
diagonalized matrices and normal forms 
there are various application-driven motivations, see e.~g. \citep{CulianezQuadrat}.

\begin{example}\label{DoublePendulum}
For instance, consider a double pendulum with lengths $\ell_1$ and $\ell_2$. 
Thus $\ell_1, \ell_2$ and $g$ are constants, that is non-zero elements of $K$ 
(for details see \citep{CulianezQuadrat}, Example 3.2.2). The linearization
of this problem leads to the system of linear partial differential equations in $\d=\frac{\d}{\d t}$, which
can be written in the matrix form with the matrix
$$M= \left[ \begin{array}{ccc}
\ell_1 \partial^2+g & 0 & -g \\
0 &            \ell_2 \partial^2+g & -g
\end{array}\right]. 
$$
Since the variable $t$ does not appear in $M$, the ground ring for the diagonalization process can be thought as of $A=\Q(\ell_1,\ell_2,g)[\d]$. Thus, indeed one can compute the Smith normal form. 

Our implementation of the diagonal form of $M$ on this example returns
\begin{footnotesize}
$$U=\left[ \begin{array}{cc} -1/g & 0 \\ -1/g & 1/g \end{array}\right] \mbox{ and }
V=\left[ \begin{array}{ccc}0 & g \ell_2  & -g \ell_2 \partial^2-g^2\\ 0 & g \ell_1  & -g \ell_1 \partial^2-g^2\\
1 & \ell_1 \ell_2 \partial^2+ g \ell_2 & -\ell_1 \ell_2 \partial^4-g \ell_1-g \ell_2 \partial^2-g^2
\end{array}\right]$$
\end{footnotesize}
such that
\begin{footnotesize}
$$ U \, M \, V = \left[ \begin{array}{ccc} 1 & 0 &  0\\0 & g \ell_1- g \ell_2 & 0\end{array}\right].
$$
\end{footnotesize}
This result agrees with results, obtained in \citep{CulianezQuadrat}.
Note, that a purely fractional method (as well as coefficient normalization procedure) will return $1$ instead of $g (\ell_1-  \ell_2)$. With our polynomial approach we obtain a polynomial matrix, which is useful for further investigations. 
In particular, in the current example we see, that setting $\ell_1= \ell_2$ implies the drop of
the rank of the Smith form from 2 to one, thus the properties of the corresponding system will change. In control theory one establishes quite different properties of the module in the non-generic case $\ell_1= \ell_2$.
\end{example}

\begin{remark}
\label{genremark}
In \citep{Genericity} the algorithm for finding so-called ``obstructions to genericity'' was derived
and discussed. A lesson learned from that paper can be applied for an implementation of 
Jacobson form as follows. It is recommended to split the algorithm (resp. the implementation)
into two parts. In the first part one computes
a diagonal matrix, where the invertibles of the ground ring are not canceled artificially. 
The second part applies the normalization on the invertibles; this part 
is trivial to achieve. Note, that our polynomial algorithm allows one to keep a close track on suspicious invertibles
due to this scheme.
\end{remark}




\begin{example}
\label{Ex3dim}
Over the first rational Weyl algebra $\Q(t)[\partial; \id, \frac{d}{dt}]$, consider the matrix 
\begin{footnotesize}
$$R=
\left[ \begin{array}{ccc} 
\partial^2 & \partial+1 &0 \\
\partial+1 & 0 & \partial^3-t^2 \partial \\
2 \partial+1 & \partial^3+\partial^2 & \partial^2
\end {array} \right].
$$
\end{footnotesize}
A typical implementation of the Jacobson normal form
returns the matrix $D = \Diag(g, 1, 1)$ together with transformation matrices 
$U, V \in \Q[t][\partial; \id, \frac{d}{dt}]^{3 \times 3}$ such that
$URV=D$. 
Below, we write down just the leading term of each matrix entry and moreover,
we write ``l.o.t.'' for ``lower order terms'' 
with respect to degree lexicographical ordering on $\Q[t][\partial; \id, \frac{d}{dt}]$.
The implementation of the Algorithm \ref{diagonalPoly} in \textsc{Singular} 
returns $D = \Diag( 2t^2d^8 + 33 \mbox{ l.o.t.}, 1,1)$. The transformation matrices are 
\begin{center}
\begin{footnotesize}
$U=\left[\begin{array}{ccc}
\frac{1}{2}t\partial^{13}+ 24 \mbox{ l.o.t. } &  \frac{1}{2}t\partial^{10}+ 19 \mbox{ l.o.t. } &   \frac{1}{2}t\partial^{11}+ 44 \mbox{ l.o.t. }\\
\frac{1}{2} &   0 &  0\\
-\frac{1}{4}\partial^5 + 2 \mbox{ l.o.t. } &  -\frac{1}{4} \partial^2 &   \frac{1}{4} + 2\mbox{ l.o.t. }
       \end{array}
\right]$, 
\end{footnotesize}
\end{center}
and
\begin{center}
\begin{footnotesize}
$V=\left[\begin{array}{ccc}
2t\partial^2 + 3 \mbox{ l.o.t. } &         2 \partial^2 &    2\partial^2 + 1\mbox{ l.o.t. } \\
-2t\partial^3 + 2 \mbox{ l.o.t. } &   -2 \partial^3 + 3 \mbox{ l.o.t. } & -2 \partial^3 \\
t \partial^8 + 28 \mbox{ l.o.t. } & \partial^8 + 11 \mbox{ l.o.t. } & \partial^8 + 16 \mbox{ l.o.t. }
       \end{array}
\right]$.
\end{footnotesize}
\end{center}
In view of \ref{unimod}, $V$ (but not $U$) is unimodular over $R_*=\Q[t][\partial; \id, \frac{d}{dt}]$.

\textsc{Janet} returns a 
matrix $\Diag(
\begin{footnotesize}
1,1, ( 279936 t^{14} + 14 \mbox{ l.o.t.})^{-1}( 279936t^{14}\partial^8+ 145 \mbox{ l.o.t.})
\end{footnotesize}
)$,       

\[
U=
\begin{footnotesize}
\left[\begin{array}{ccc}
1 & 0 & 0\\
(6t^2 + 2 \mbox{ l.o.t.})^{-1}(\partial^2 + 1\mbox{ l.o.t.})& 
(6t^2 + 2 \mbox{ l.o.t.})^{-1}(\partial^3 + 3\mbox{ l.o.t.})& (6t^2 + 2 \mbox{ l.o.t.})^{-1}\\
u_{31}  & u_{32} & u_{33} \end{array}
\right], \end{footnotesize}
\]
where 
\begin{footnotesize}
$g = (559872 t^{14} + 14 \mbox{ l.o.t.}), u_{31} = g^{-1}(-279936{t}^{14}{\partial}^{9} + 158 \mbox{ l.o.t.}), u_{32} = g^{-1}(279936{t}^{14}{\partial}^{10} + 182 \mbox{ l.o.t.}), u_{33}  = g^{-1}(279936{t}^{14}{\partial}^{7} + 127 \mbox{ l.o.t.})$
\end{footnotesize}. 
The right transformation matrix $V = $
 \begin{center}
 \begin{footnotesize}
 $\left[\begin{array}{ccc}
 1 & \frac{1}{2}\partial^6+ 15 \mbox{ l.o.t. }& 
 (279936t^{14}+14\mbox{ l.o.t. })^{-1}(46656{t}^{12}\partial^{7} + 110\mbox{ l.o.t. }) \\
 \partial+ 1 \mbox{l.o.t.} & -\frac{1}{2}\partial^7+ 15 \mbox{ l.o.t. } & 
 (-1679614t^{16}+16\mbox{ l.o.t. })^{-1}(279936{t}^{14}\partial^{8} + 138\mbox{ l.o.t. })\\
 0 & 1 & (6t^2 + 2 \mbox{ l.o.t. })^{-1} (2\partial^2+ 1 \mbox{ l.o.t })
        \end{array}
 \right]$.
 \end{footnotesize}
 \end{center}
\end{example}

\begin{example}
Consider the matrix from the Example \ref{Ex3dim}, replacing $\partial$ by $S$, 
the forward shift operator in the first rational shift algebra in $t, s$. Then 
the diagonal form, computed with our algorithm is
$\Diag(t^{12} S^{8} + 101$ l.o.t.$,1,1)$.
Notably, the leading coefficient in $t$ factorizes completely. Transformation matrices are, as expected, more complicated as in the Example \ref{Ex3dim}. $U$ has only three entries of length bigger than 1; their lengths are 113, 116, 150. In the matrix $V$, the lengths of entries are
22, 11, 58, 20, 14, 60, 26, 17, 64 with degree in $S$ up to 7. Coefficients, having more than 7 digits appear only in one entry, and grow up to 12 digits. The situation in the first rational difference algebra is similar, as a reader can see by computing with our implementation. We have computed all the examples from this paper in the shift and difference settings as well.
\end{example}

\begin{example}
Let $R=\Q(y,x)[\partial; \id, \frac{d}{dx}]$ and thus $R_*=\Q[y,x][\partial; \id, \frac{d}{dx}]$.
The matrix $M$ below comes from the system of partial differential equations. With our
algorithm we obtain transformation matrices and a diagonal one: 
\[
M=\left[ \begin{array}{cc} y^2 \partial^2 + \partial +1 & 1 \\ x\partial & x^2 \partial^2+\partial + y
\end{array}\right], 
\]
\[
\left[ \begin{array}{cc} -x^2 \partial^2-\d-y & 1\\1& 0 \end{array}\right] \, M \,
\left[ \begin{array}{cc} 1 & 0 \\ -y^2 \d^2-\d-1 & 1 \end{array}\right] =
\left[ \begin{array}{cc} g & 0 \\ 0 & 1 \end{array}\right],
\]
where $g=-y^2x^2 \d^4-x^2 \d^3-x^2 \d^2-y^2 \d^3+x\d+(-y^3-1)\d^2+(-y-1)\d-y$.

If we consider $M \in \Z_2(y,x)[\partial; \id, \frac{d}{dx}]^{2 \times 2}$,
we obtain the single example from \citep{Middeke}. Then the rational form of our result
is exactly the result obtained in \citep{Middeke}, namely $\Diag(1,-\tfrac{g}{x^2y^2} \mod 2)$.
Note, that in our method no computations with $4\times 4$ matrices as in \citep{Middeke} are needed. As demonstrated, our implementation works over finite fields as well. And, as before, the right transformation
matrix $V$ is unimodular even over $R_* = \Q[y,x][\partial; \id, \frac{d}{dx}]$.
\end{example}

As we have seen, in all the examples above the right transformation matrix $V$ was indeed unimodular over $R_*$. We observe this phenomenon for even more examples over Weyl and shift algebras. 
\begin{conjecture}
\label{conjUnimod}
Let $A_*$ be a $G$-algebra and $A=\Quot(A_*)$. Moreover, let $R=A[\partial; \sigma, \delta]$, such that $R_* = A_*[\partial; \sigma, \delta]$ is a $G$-algebra. For a matrix $M\in R^{p\times p}$ there exist square matrices $U,V,D$ with entries from $R_*$, such that $UMV = D$, where $D$ is diagonal and $U,V$ unimodular over $R$. If $D$ has only one polynomial non-constant entry, then $V$ can be chosen to be unimodular over $R_*$.
\end{conjecture}

\noindent
\textbf{Application}. Over $R$, the decomposition as above can be applied as follows. We start with a system of equations $M \omega = 0$ in unknown functions $\omega = (\omega_1,\ldots,\omega_p)$. Since $U$ and $V$ are unimodular over $R$ and $UMV=\Diag(d_{11},\ldots,d_{pp})$, we obtain a decoupled system $\{ d_{ii} z_i = 0 \}$, where $z = V^{-1} \omega$, which is 
equivalent to $M \omega = 0$ over $R$. Note, that $d_{ii}=0$ is possible, then one calls $z_i$ a free variable of the system in the literature (e.~g. in \citep{Eva05}).

Let us analyze what can be done over $R_*$. Suppose, that 
$V$ is unimodular over $R_*$. Then $U M \omega = 0 \Leftrightarrow D V^{-1} \omega = 0$. However, since $U$ is not unimodular over $R_*$, we have implication $M \omega = 0 \Rightarrow UM \omega=0$ only. Let $T$ be a matrix, such that $TU=\id_R$, then, by
a reasoning, similar to Lemma \ref{inMod} there exists 
a diagonal matrix $Q=\Diag(\ldots,q_{ii},\ldots)$ such that $Q$ resp. $QT$ have
with entries from $A_*$ resp. $R_*$.
For simplicity, assume that $A_*$ is commutative.
Denote by $S$ the multiplicatively closed set, generated by $q$, the least common multiple of $\{ q_{ii} \}$. If $S$ happens to be an Ore set in $R_*$, then the localization $S^{-1} R_*$ exists and $U$ will be unimodular over $S^{-1} R_*$. Further computations happen in different branches: first in the generic $S^{-1} R_*$, where by $UM \omega = 0 \Rightarrow M \omega = 0$ and then in the case, determined by the relation $q = 0$. 
In the latter, one can apply the algorithm \texttt{Genericity} from \citep{Genericity}, which delivers a disjoint decomposition of the set of zeros of $q$ into locally closed sets $L_i$. One can proceed with analysis of systems $UMw=0$ along $L_i$ and obtain special solutions on each $L_i$. 
This shows, that the left transformation matrix $U$ in this setting carries essential information about the so-called singularities of a system. Note, that working over $R$ we compute only
generic information, while following the polynomial strategy over $R_*$ allows us to make a complete description of the system.

Clearly the decoupling, provided by a diagonal form, is of big importance for solving systems of operator equations with rational coefficients and for the structural analysis, performed in the algebraic system and control theory (see e.~g. Theorem 8 of \citep{Eva05}).

\section{Conclusion and Future Work}

Indeed, this paper is a part of a general program on providing effective
computations within Ore localized $G$-algebras. Notably, polynomial strategy,
which we described in details for the case of one polynomial variable, 
is one of the key elements of the program. There is ongoing work on
the implementation of Gr\"obner bases for Ore localized $G$-algebras under
the codename \textsc{Singular::Locapal}. \\

Polynomial strategy brings us several advantages in practical computations.
One of them is the generality of the overall approach. Namely, as soon
as there is an implementation of Gr\"obner bases for modules (and hence
syzygies) over a $G$-algebra $A$, under some mild assumptions we are able to work
effectively with Ore localization $A_{B^*}$ of $A$ with respect to a multiplicatively closed
Ore set $B^*$, where $B$ is a suitable $G$-subalgebra of $A$ (cf. Theorem \ref{BOre}).

The question, whether direct computations with fractions of $A_{B^*}$ will be always
outperformed by the polynomial strategy, is still open. Consider, for instance, the situation,
where the input matrix $M$ is given already with rational non-commutative coefficients. Then bringing $M$
to the fraction-free form is already a nontrivial operation (as soon as we work with non-commutative algebra), as indicated e.~g. in the proof of Lemma \ref{inMod}. 

In our opinion the answer to the above question depends both on the algebra $A_{B^*}$ and on the
presentation matrix $M$. However, in general nontrivial computation directly using fractions in
the algorithm might cause the appearance of enormous coefficients, as several
examples demonstrate. We want to stress, that these examples have not been
specially selected for this purpose; instead, we picked a couple of them from
a bigger family of examples. In our opinion, this phenomenon is quite
ubiquitous. \\

Our implementation of the Jacobson normal form will be developed further
to provide a user with the possibility to compute in more general algebras. At the
moment, the stable version of the library \citep{Jacobsonlib} supports first Weyl, shift and difference algebras. Investigation of normal forms over non-simple domains (as in \ref{cexShift}, \ref{qWeyl}) is an important future task.

Middeke \citep{Middeke} has reported, that the classical algorithm, computing Jacobson form
of a matrix over the Weyl algebra over a differential field is polynomial-time.
However, it seems to us (due to polynomial strategy approach), that the 
subalgebra of invertible elements must be involved in the complexity analysis.
Perhaps one should consider different models for studying complexity, since
experience with practical applications suggests, that the important role, played
by the coefficient arithmetics (which is not the arithmetics over a numerical
field anymore!) must be appropriately reflected in the overall complexity. Otherwise
the complexity of operations over the skew field of invertible elements remains hidden. \\

Recently, Mark Giesbrecht and George Labahn suggested 
the use of another technique from \citep{Kaltofen}, namely the randomization. Starting with a matrix $M$, 
one multiplies $M$ with random square (hence unimodular) matrices from both sides, in order to reduce the number of iterations in the Algorithms 1 and 2. Some experiments confirm that this might be generalized to the setting of localized $G$-algebras. However, the computations become much harder in practice due to increased size of polynomials to deal with. This is another reason for our proposal to investigate the different notions of complexity of operations over skew fields.

\section*{Acknowledgments}

The authors are very grateful to Eva Zerz and Hans Sch\"onemann for their advice on
numerous aspects of the problems, treated in this article. We thank to Daniel Robertz, Johannes Middeke and Howard Cheng for explanations about respective implementations. 


\begin{thebibliography}{10}

\bibitem{ApelDiss}
J.~Apel.
\newblock Gr{\"o}bnerbasen in nichtkommutativen {A}lgebren und ihre
  {A}nwendung.
\newblock {\em Dissertation, Universit{\"a}t Leipzig}, 1988.

\bibitem{Involutlib}
M.~Becker, V.~Levandovskyy, and O.~Yena.
\newblock A \textsc{Singular} 3.0 library for computations and operations with
  involutions \texttt{involut.lib}, 2003.
\newblock http://www.singular.uni-kl.de.

\bibitem{BCL06}
B.~Beckermann, H.~Cheng, and G.~Labahn.
\newblock Fraction-free row reduction of matrices of skew polynomials.
\newblock In T.~Mora, editor, {\em Proc. of the International Symposium on
  Symbolic and Algebraic Computation (ISSAC'02)}, pages 8--15. ACM Press, 2002.

\bibitem{Robertz}
Y.~A. Blinkov, C.~F. Cid, V.~P. Gerdt, W.~Plesken, and D.~Robertz.
\newblock The {MAPLE} package "janet": {II}. {L}inear {P}artial {D}ifferential
  {E}quations.
\newblock In {\em Proceedings of the 6th International Workshop on Computer
  Algebra in Scientific Computing}, pages 41--54, 2003.
\newblock http://wwwb.math.rwth-aachen.de/Janet.

\bibitem{BGV}
J.~Bueso, J.~G{\'o}mez-Torrecillas, and A.~Verschoren.
\newblock {\em Algorithmic methods in non-commutative algebra. Applications to
  quantum groups}.
\newblock Kluwer Academic Publishers, 2003.

\bibitem{CL07}
H.~Cheng and G.~Labahn.
\newblock Modular computation for matrices of {O}re polynomials.
\newblock In {\em Computer Algebra 2006: Latest Advances in Symbolic
  Algorithms}, pages 43--66, 2007.

\bibitem{Chyzak}
F.~Chyzak.
\newblock Gr{\"o}bner bases, symbolic summation and symbolic integration.
\newblock In B.~Buchberger and F.~Winkler, editors, {\em 33 years of
  Gr{\"o}bner bases}, pages 32--60. Cambridge University Press, LMS LNS 251,
  1998.

\bibitem{CQR}
F.~Chyzak, A.~Quadrat, and D.~Robertz.
\newblock \textsc{OreModules}: A symbolic package for the study of
  multidimensional linear systems.
\newblock In J.~Chiasson and J.-J. Loiseau, editors, {\em Applications of
  Time-Delay Systems}, pages 233--264. Springer LNCIS 352, 2007.
\newblock http://wwwb.math.rwth-aachen.de/OreModules.

\bibitem{ChyzakSalvy}
F.~Chyzak and B.~Salvy.
\newblock Non--commutative elimination in {O}re algebras proves multivariate
  identities.
\newblock {\em J. Symbolic Computation}, 26(2):187--227, 1998.

\bibitem{Cohn}
C.~Cohn.
\newblock {\em Free Rings and their Relations}.
\newblock Academic Press, 1971.

\bibitem{CulianezQuadrat}
G.~Culianez and A.~Quadrat.
\newblock Formes de {H}ermite et de {J}acobson: implementations et
  applications.
\newblock Technical report, INRIA Sophia Antipolis, 2005.

\bibitem{DCL08}
P.~Davies, H.~Cheng, and G.~Labahn.
\newblock Computing {P}opov form of general {O}re polynomial matrices.
\newblock In {\em Proceedings of the Milestones in Computer Algebra (MICA)
  Conference}, pages 149--156, 2008.

\bibitem{GML}
J.~I. Garc\'{\i}a~Garc\'{\i}a, J.~Garc\'{\i}a~Miranda, and F.~J. Lobillo.
\newblock Elimination orderings and localization in {PBW} algebras.
\newblock {\em Linear Algebra Appl.}, 430(8-9):2133--2148, 2009.

\bibitem{GTZ}
P.~Gianni, B.~Trager, and G.~Zacharias.
\newblock Gr{\"o}bner bases and primary decomposition of polynomial ideals.
\newblock {\em J. Symbolic Computation}, 6(2-3):149--167, 1988.

\bibitem{Plural}
G.-M. Greuel, V.~Levandovskyy, and H.~Sch{\"o}nemann.
\newblock {\textsc{Plural}. A \textsc{Singular} 3.0 Subsystem for Computations
  with Non--commutative Polynomial Algebras. Centre for Computer Algebra,
  University of Kaiserslautern}, 2006.
\newblock http://www.singular.uni-kl.de.

\bibitem{Singular}
G.-M. Greuel, G.~Pfister, and H.~Sch{\"o}nemann.
\newblock \textsc{Singular} 3.1. a {C}omputer {A}lgebra {S}ystem for
  {P}olynomial {C}omputations. {C}entre for {C}omputer {A}lgebra, {U}niversity
  of {K}aiserslautern, 2009.
\newblock http://www.singular.uni-kl.de.

\bibitem{IlchmannMehrmann}
A.~Ilchmann and V.~Mehrmann.
\newblock A behavioral approach to time-varying linear systems. {I}: {G}eneral
  theory.
\newblock {\em SIAM J. Control Optim.}, 44(5):1725--1747, 2006.

\bibitem{Ilchmann}
A.~Ilchmann, I.~N{\"u}rnberger, and W.~Schmale.
\newblock Time-varying polynomial matrix systems.
\newblock {\em Int. J. Control}, 40:329--362, 1984.

\bibitem{Spanier}
M.~A. Insua.
\newblock {\em Varias perspectives sobre las bases de {G}r{\"o}bner: {F}orma
  normal de {S}mith, {A}lgoritme de {B}erlekamp y {\'a}lgebras de {L}eibniz}.
\newblock PhD thesis, Universidade de Santiago de Compostela, Spain, 2005.

\bibitem{Jacobson}
N.~Jacobson.
\newblock {\em The Theory of Rings}.
\newblock American Mathematical Society, 1943.

\bibitem{Kaltofen}
E.~Kaltofen, M.~S. Krishnamoorthy, and B.~D. Saunders.
\newblock Mr. {S}mith goes to {L}as {V}egas: {R}andomized parallel computation
  of the {S}mith normal form of polynomial matrices.
\newblock In J.~H. Davenport, editor, {\em Proc. EUROCAL '87}, volume 378 of
  {\em LNCS}, pages 317--322. Springer, 1989.

\bibitem{Kredel}
H.~Kredel.
\newblock {\em Solvable polynomial rings}.
\newblock Shaker, 1993.

\bibitem{Viktor}
V.~Levandovskyy.
\newblock {\em Non-commutative Computer Algebra for polynomial algebras:
  Gr{\"o}bner bases, applications and implementation}.
\newblock PhD thesis, Universit{\"a}t Kaiserslautern, 2005.
\newblock http://kluedo.ub.uni-kl.de/volltexte/2005/1883/.

\bibitem{LS03}
V.~Levandovskyy and H.~Sch{\"o}nemann.
\newblock Plural --- a computer algebra system for noncommutative polynomial
  algebras.
\newblock In {\em Proc. of the International Symposium on Symbolic and
  Algebraic Computation (ISSAC'03)}, pages 176 -- 183. ACM Press, 2003.
\newblock http://doi.acm.org/10.1145/860854.860895.

\bibitem{Genericity}
V.~Levandovskyy and E.~Zerz.
\newblock Obstructions to genericity in study of parametric problems in control
  theory.
\newblock In H.~Park and G.~Regensburger, editors, {\em Gr{\"o}bner Bases in
  Control Theory and Signal Processing}, volume~3 of {\em Radon Series Comp.
  Appl. Math}, pages 191--214. Walter de Gruyter \& Co., 2007.
\newblock http://arxiv.org/abs/0708.2078.

\bibitem{Frank}
F.~L{\"u}beck.
\newblock On the computation of elementary divisors of integer matrices.
\newblock {\em J. Symbolic Computation}, 33(1):57--65, 2002.

\bibitem{McConnellRobson}
J.~McConnell and J.~Robson.
\newblock {\em Noncommutative Noetherian rings.}
\newblock AMS, 2001.

\bibitem{Middeke}
J.~Middeke.
\newblock A polynomial-time algorithm for the {J}acobson form for matrices of
  differential operators.
\newblock Technical Report 2008-13, {RISC}, J. Kepler University Linz, 2008.

\bibitem{Newman}
M.~Newman.
\newblock {\em Integral matrices.}
\newblock Academic Press, 1972.

\bibitem{Jacobsonlib}
K.~Schindelar and V.~Levandovskyy.
\newblock A \textsc{Singular} 3.1 library with algorithms for {S}mith and
  {J}acobson normal forms \texttt{jacobson.lib}, 2009.
\newblock http://www.singular.uni-kl.de.

\bibitem{Eva05}
E.~Zerz.
\newblock An algebraic analysis approach to linear time-varying systems.
\newblock {\em IMA J. Math. Control Inf.}, 23(1):113--126, 2006.

\bibitem{Eva}
E.~Zerz.
\newblock State representations of time-varying linear systems.
\newblock In H.~Park and G.~Regensburger, editors, {\em Gr{\"o}bner Bases in
  Control Theory and Signal Processing}, volume~3 of {\em Radon Series Comp.
  Appl. Math}, pages 235--251. Walter de Gruyter \& Co., 2007.

\end{thebibliography}

\end{document}